\documentclass[12pt,reqno]{amsart}
\usepackage[usenames,dvipsnames]{xcolor}
\usepackage{amsmath, amsfonts, epsfig, amssymb, graphics}
\usepackage[colorlinks=true,citecolor=cyan]{hyperref}
\usepackage[mathscr]{euscript}

\newcommand{\pref}[1]{{\rm (\ref{#1})}}
\def\mb{\mathbf}

\def\dinv{\mathrm{dinv}}
\def\area{\mathrm{area}}
\def\composition{\gamma}

\def\Ring{\mathcal{R}}
\def\Altid{\mathcal{A}}
\def\Alt{\mathscr{A}}
\def\Ideal{\mathcal{I}}
\def\Inv{\mathscr{I}}
\def\Base{\mathscr{B}}
\def\C{\mathbb{C}}
\def\Harm{\mathscr{H}}
\def\Qsym{\mathscr{Q}}
\def\Desc{{\rm Desc}}

\def\R{\mathbb{R}}
\def\Z{\mathbb{Z}}
\def\S{{\mathbb S}}
\def\N{{\mathbb N}}

\newcommand{\sign}{{\rm sign}}

\def\mbf#1{{\mathbf #1}}
\def\part#1#2{\bleu{\rouge{\{} \{#1\}\rouge{,}\{#2\}\rouge{\}}}}
\newcommand{\scalar}[2]{{\langle#1,#2 \rangle}}
\newcommand{\newatop}[2]{\genfrac{}{}{0pt}{}{#1}{#2}}

 \DeclareMathOperator\bijection{\smash{\buildrel\sim\over\longrightarrow}}
 
\def\auteur#1{{\sc #1}}
\def\titreref#1{{\em #1}}
\def\vol#1{{\bf #1}}
\def\defn#1{\bleu{{\bf #1}}}

\newcommand{\del}[1]{\partial#1}
\newcommand{\Park}{\mathcal{P}}

\def\Dyck#1#2{\mathcal{D}^{(#1)}_{#2}}
\def\charac{\raise 2pt\hbox{\begin{math}\chi\end{math}}}

\newdimen\carrelength
\def\jcarre{\jaune{\linethickness{\carrelength}\line(1,0){.85}}}

\def\bleu{\textcolor{blue}}
\def\rouge{\textcolor{red}}
\def\jaune{\textcolor{yellow}}
\def\bleupale#1{{\color{SkyBlue} #1}}
\def\vert#1{{\color{LimeGreen} #1}}

\newtheorem{proposition}{\bleu{Proposition}}
\newtheorem{lemma}{\bleu{Lemma}}
\newtheorem{conjecture}{\bleu{Conjecture}}
\newtheorem{open}{\bleu{Open problem}}

\begin{document} 

\title[Fuss-Catalan Combinatorics]{\bleu{Combinatorics of  \lowercase{$r$}-Dyck paths,\\ \lowercase{$r$}-Parking functions,\\ and the \lowercase{$r$}-Tamari lattices}}
\author[F.~Bergeron]{F. Bergeron}
\address{D\'epartement de Math\'ematiques, UQAM,  C.P. 8888, Succ. Centre-Ville, 
 Montr\'eal,  H3C 3P8, Canada.}\date{January 2011. This work was supported by NSERC-Canada. Parts of it has appeared in \cite{trivariate}.}
 \email{bergeron.francois@uqam.ca, lfprevilleratelle@hotmail.com}
\maketitle
\begin{abstract}
This is a survey-like paper presenting some of the recent combinatorial considerations on 
$r$-Dyck paths, $r$-Parking functions, and especially the $r$-Tamari lattices. Giving a better understanding of the combinatorial interplay between these objects has become important in view of their (conjectural) role in the description of the graded character of the $\S_n$-modules of bivariate and trivariate diagonal coinvariant spaces for the symmetric group.
\end{abstract}
 \parskip=0pt

%%%%%%%%%%%%%%%%%%%%%%%%%%   
{ \setcounter{tocdepth}{1}\parskip=0pt\footnotesize \tableofcontents}
\parskip=8pt

%%%%%%%%%%%%%%%%%%%%%%%%%%%%%%%%%%%%%%%%%%%%
%%%%%%%%%%%%%%%%%%%%%%%%%%%%%%%%%%%%%%%%%%%%
\section{Introduction}
In this paper, our intent is to describe the interplay between various combinatorial constructions  that have recently been considered in relation with the study of ``diagonal coinvariant spaces'' for the symmetric group.  This story started back in the early 1990's with the seminal work of Garsia and Haiman on the bivariate diagonal coinvariant space, and recently unfolded in a new direction with the work of the author on the corresponding trivariate case. Right at the beginning, 20 years ago, both $r$-Dyck paths and $r$-parking functions where involved in interesting ways in setting up formulas for the character of the relevant spaces, in the bivariate situation. Since then, many new developments have eloquently confirmed the role of these combinatorial object in this context. 
Still the saga of the study of these spaces is far from over, and one the main remaining open problem is  to prove the the ``Shuffle Conjecture'' of \cite{HHLRU}. This conjecture gives an entirely combinatorial description of the graded character for the bivariate case, all in terms of $r$-parking functions. 

Very recently, it has become apparent (see \cite{trivariate}) that the $r$-Tamari order is a main player in the extension of this saga to the trivariate case.  In particular, it allows for an extension to the trivariate  case of the Shuffle Conjecture. It is also worth mentioning that setting up this new conjecture has lead to interesting, and hard to prove, new combinatorial identities (see \cite{melou, chapuy}). It seems interesting to present, in one unified text, the various combinatorial developments that play a role in this story, as well as the necessary background so the story be mostly self-contained.
Several new research directions are discussed.

\section{Dyck paths and Tamari order}
Catalan numbers
   \begin{eqnarray}
         \bleu{C_n}&:=&\bleu{\frac{1}{n+1}\binom{2\,n}{n}},\label{formule}\\
                            &=& 1, 2, 5, 14, 42, 132, 429, 1430, 4862, 16796, 58786, \ \nonumber\ldots
    \end{eqnarray}
 have a long an interesting history (see \cite{koshy}) going back to Ming and Euler, in the mid 18th century (a long time before Catalan). 
  \begin{figure}[!ht]
\begin{center}\scalebox{1}{\includegraphics{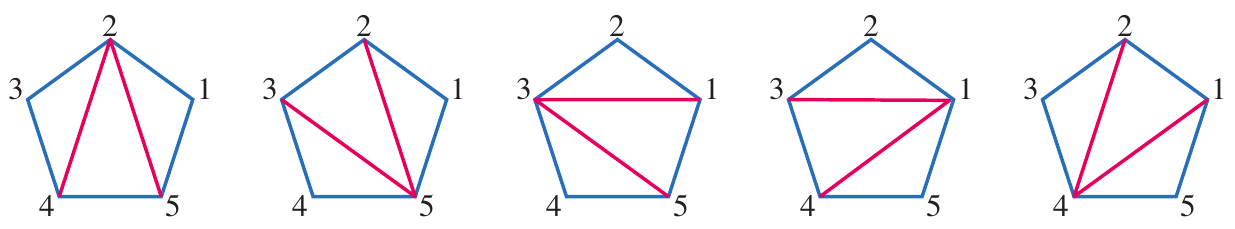}}
 \end{center}\vskip-20pt
\caption{Triangulations of the pentagon.}\label{fig_triangulation}
\end{figure}
 For Euler they appeared as the number of triangulations of a polygon (see Fig.~\ref{fig_triangulation}), and he gave both a formula, equivalent to \pref{formule}, and the generating function
 \begin{equation}\label{genfonct}
   \bleu{\frac{1-2\,x-\sqrt{1-4\,x}}{2\,x^2}=1+2\,x+5\,x^2+14\,x^3+42\,x^4+132\,x^5+\ \ldots}.
\end{equation}
Still before Catalan, these numbers where generalized by Fuss (\cite{fuss}) by considering other polygon decompositions. Following a growing tradition, we say that the resulting numbers
   \begin{equation}\label{fusscatalan}
         \bleu{C_n^{(r)}}:=\bleu{\frac{1}{r\,n+1}\binom{(r+1)\,n}{n}},
    \end{equation}
are the \defn{Fuss-Catalan numbers}. There is a very extensive classical literature on a large number of combinatorial interpretations of these integers (see \cite{koshy}); and many new interesting aspects are still being discovered. One especially interesting direction is a connection with the study of Shi arrangements, and their analogs for other Coxeter groups (see \cite{armstrong, athanasiadis}).

Observe that  $C_n^{(r)}$ may be considered as a degree $n-1$ polynomial in  the ``parameter'' $r$. Indeed, directly from  Formula \pref{fusscatalan},we see that
  \begin{equation}\label{fusscatalanpol}
         \bleu{C_n^{(r)}}=\bleu{\frac{1}{n!} (r\,n+2)(r\,n+3)\cdots (r\,n+n)}.
    \end{equation}
so that
\begin{eqnarray*}
   \bleu{C_1^{(r)}}&=&\bleu{1} ,   \\
  \bleu{C_2^{(r)}}&=&\bleu{(2\,r+2)/2} ,   \\
  \bleu{C_3^{(r)}}&=&\bleu{ (3\,r+2)(3\,r+3)/6} ,   \\
 \bleu{C_4^{(r)}}&=&\bleu{ (4\,r+2)(4\,r+3)(4\,r+4)/24} ,   \\
 \bleu{C_5^{(r)}}&=&\bleu{ (5\,r+2)(5\,r+3)(5\,r+4)(5\,r+5)/120} ,\  \ldots \\
\end{eqnarray*}
At any $r\in\N$, these polynomial evaluate to positive integers.
Thus, with $r=2$, we get the sequence of integers
   $$ 1, 3, 12, 55, 273, 1428, 7752, 43263, 246675, 1430715,\ldots $$
The generating series 
\begin{equation}\label{defn_genseries}
    \bleu{\mathcal{C}^{(r)}(x)}:=\bleu{\sum_{n=0}^\infty \frac{1}{rn+1}\binom{(r+1)\,n}{n}\, x^n},
  \end{equation}
for Fuss-Catalan numbers, satisfies the algebraic equation 
   \begin{equation}
       \bleu{Y(x)=1+x\,Y(x)^{r+1}}.
   \end{equation} 
Equivalently, the Fuss-Catalan numbers satisfy the recurrence
  \begin{equation}\label{fussrec}
      \bleu{C_{n+1}^{(r)}}=\bleu{\sum_{n=n_1+n_2+\ldots+n_{r+1}} C_{n_1}^{(r)} C_{n_2}^{(r)}\cdots  C_{n_{r+1}}^{(r)}},
  \end{equation}
 with initial condition $C_0^{(r)}=C_1^{(r)}=1$.
   
We here consider Fuss-Catalan numbers as enumerators of \bleu{$r$-Dyck paths}.  Recall that
these are the paths in $\N\times \N$,  i.e.: sequences of ``south'' and ``east'' steps\footnote{We chose this representation to simplify our later discussion.}, that 
\begin{itemize}\itemsep=4pt
   \item start at $(0,n)$ and end at $(r\,n,0)$, and
   \item  stay below the line $y=({-1}/{r})\,x+n$, of slope $-1/r$.
\end{itemize}
For instance, we have the $2$-Dyck path of Figure~\ref{fig1}.
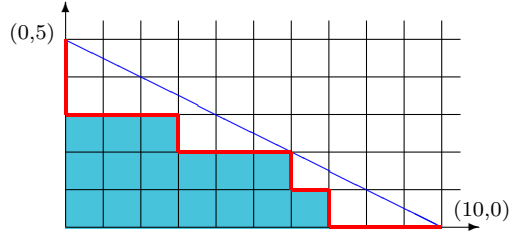
\begin{figure}[ht]
\setlength{\unitlength}{5mm}
\setlength{\carrelength}{4mm}
\def\jcarre{\put(.1,-.48){\jaune{\linethickness{\carrelength}\line(1,0){.75}}}}
\def\palecarre{\bleupale{\linethickness{\unitlength}\line(1,0){1}}}
\begin{center}
\begin{picture}(11,7)(0,0)
\put(0,.45){\multiput(0,2)(1,0){3}{\palecarre}
                   \multiput(0,1)(1,0){6}{\palecarre}
                   \multiput(0,0)(1,0){7}{\palecarre}}
\multiput(0,1)(0,1){5}{\line(1,0){10.5}}
\multiput(1,0)(1,0){10}{\line(0,1){5.5}}
 \put(0,0){\vector(1,0){11}}
 \put(0,0){\vector(0,1){6}}
\put(-1.5,5){$\scriptstyle(0,5)$}
 \put(0,5){\bleu{\line(2,-1){10}}}
  \linethickness{.5mm}
\put(0,5){\rouge{\line(0,-1){1}}}
\put(0,4){\rouge{\line(0,-1){1}}}
\put(0,3){\rouge{\line(1,0){3}}}
\put(3,3){\rouge{\line(0,-1){1}}}
\put(3,2){\rouge{\line(1,0){3}}}
\put(6,2){\rouge{\line(0,-1){1}}}
\put(6,1){\rouge{\line(1,0){1}}}
\put(7,1){\rouge{\line(0,-1){1}}}
\put(7,0){\rouge{\line(1,0){3}}}
\put(10.3,0.3){$\scriptstyle(10,0)$}
\end{picture}\end{center}
\caption{The $2$-Dyck path encodes as $00367$.}
\label{fig1}
\end{figure}
More precisely, a $r$-Dyck path $\alpha$ is a sequence 
   $$\alpha=(p_0,p_1,\ldots, p_{N})$$
of $N=r\,n+1$ points $p_i=(x_i,y_i)$, in $\N\times \N$, such that
\begin{itemize}\itemsep=4pt
\item $p_0=(0,n)$ and $p_N=(r\,n,0)$,
\item $p_{i+1}=p_i+(-1,0)$ (a \defn{south step}), or $p_{i+1}=p_i+(0,1)$ (an \defn{east step}), and
\item $x_i\leq r\,(n-y_i)$, for all $1\leq i<n$.
\end{itemize}
We say that $n$ is the \defn{height} of the path $\alpha$, and we denote by $\bleu{\Dyck{r}{n}}$ the  set of $r$-Dyck paths of height $n$.
It is well known  that the number of $r$-Dyck paths of height $n$ is given by the Fuss-Catalan number.
%Whenever $r$ is clear from the context, we simply say that we have a Dyck path rather than a $r$-Dyck path. 
%This convention will also apply in other similar situations. 

For us, it will be practical to also describe Dyck paths in terms of the number of ``cells'' that lie to the left of south steps. More precisely, we bijectively encode  height $n$ Dyck paths in terms of sequences $a_1a_2 \cdots a_n$, with $a_i$ the $x$-coordinate of the leftmost point on the path, among those  that lie at height $i$.  In this encoding, $\alpha=a_1a_2 \cdots a_n$ corresponds to a $r$-Dyck path if and only if
\begin{itemize}\itemsep=4pt
   \item[(1)] $\bleu{0\leq a_1\leq a_2\leq \ldots \leq a_n}$, and
   \item[(2)] for each $i$, we have $\bleu{a_i\leq r(i-1)}$.
\end{itemize}
For example, the $2$-Dyck paths of {height} $3$ are now encoded as 
  \begin{displaymath}  000,001,002,003,004,011,012,013,014,022,023,024. \end{displaymath}
In such a description, each integer $k$, between $0$ and $n\,(r-1)$, occurs with some \defn{multiplicity}\footnote{possibly equal to 0.} $n_k$. Removing the $0$-multiplicities, we obtain the  \defn{(multiplicity) composition} $\composition(\alpha)$ of the sequence $\alpha$. Hence, we have
$$\begin{array}{llll}
  \composition(000) = 3, & \composition(001) = 21, & \composition(011) = 12, & \composition(002) = 21\\[4pt]
  \composition(012) = 111, & \composition(022) = 12, & \composition(003) = 21, & \composition(013) = 111\\[4pt]
  \composition(023) = 111, & \composition(004) = 21, & \composition(014) = 111, &\composition(024) = 111.\\
\end{array}$$
By extension, we also say that $\composition(\alpha)$ is the composition of the corresponding Dyck path. Clearly $\composition(\alpha)$ is a composition of $n$ (it sums up to $n$).

The \defn{area}, defined for $\alpha$ in $\Dyck{r}{n}$ as
    \begin{equation}\label{defn_area}
          \bleu{\area(\alpha)}:=\bleu{r\,\binom{n}{2}-\sum_{i=1}^n a_i},
      \end{equation}
is one of the well known interesting statistics on Dyck paths. Summing over the whole set $\Dyck{r}{n}$, we obtain the $q$-enumerator polynomial
   \begin{equation}
      \bleu{ C_{n+1}^{(r)}(q)}:=\bleu{\sum_{\alpha} q^{\area(\alpha)}},
   \end{equation}
 which satisfies the recurrence
     \begin{equation}\label{qfussrec}
        \bleu{C_{n+1}^{(r)}(q)} =\bleu{\sum_{n=n_1+n_2+\ldots+n_{r+1}} \rouge{q^{r\,n_1}} C_{n_1}^{(r)} \rouge{q^{(r-1) \,n_2}}C_{n_2}^{(r)}\cdots  C_{n_{r+1}}^{(r)}},
  \end{equation}
 with initial condition $C_0^{(r)}=C_1^{(r)}=1$.
Equivalently,  the generating series 
\begin{equation}\label{defn_genseriesq}
    \bleu{\mathcal{C}^{(r)}(x)}:=\bleu{\sum_{n=0}^\infty C_{n+1}^{(r)}(q)\, x^n},
  \end{equation}
for these \defn{$q$-Fuss Catalan} numbers, satisfies the $q$-algebraic equation
   \begin{equation}
       \bleu{Y(x)=1+x\,Y(\rouge{q^r}\,x)\,Y(\rouge{q^{r-1}}\,x)\cdots Y(\rouge{q}\,x)Y(x) }.
   \end{equation}
For $r=1$, we have
\begin{eqnarray*}
C_{1}(q)&=&1\\
C_{2}(q)&=&1+q\\
C_{3}(q)&=&1+2\,q+{q}^{2}+{q}^{3}\\
C_{4}(q)&=&1+3\,q+3\,{q}^{2}+3\,{q}^{3}+2\,{q}^{4}+{q}^{5}+{q}^{6}\\
C_{5}(q)&=&1+4\,q+6\,{q}^{2}+7\,{q}^{3}+7\,{q}^{4}+5\,{q}^{5}+5\,{q}^{6}+3\,{q}^{7}+2\,{q}^{8}+{q}^{9}+{q}^{10}
\end{eqnarray*}
\setlength{\unitlength}{5mm}
\newdimen\carrelength
\setlength{\carrelength}{4.1mm}
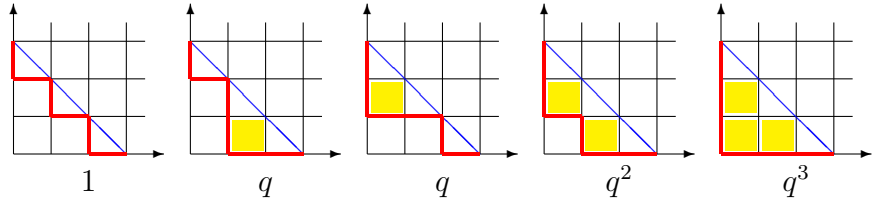
\begin{figure}[ht]
$$\begin{array}{ccccc}
     \begin{picture}(4,4)(0,0)
\multiput(0,1)(0,1){3}{\line(1,0){3.5}}
\multiput(1,0)(1,0){3}{\line(0,1){3.5}}
 \put(0,0){\vector(1,0){4}}
 \put(0,0){\vector(0,1){4}}
 \put(0,3){\bleu{\line(1,-1){3}}}
  \linethickness{.5mm}
\put(0,3){\rouge{\line(0,-1){1}}}
\put(0,2){\rouge{\line(1,0){1}}}
\put(1,2){\rouge{\line(0,-1){1}}}
\put(1,1){\rouge{\line(1,0){1}}}
\put(2,1){\rouge{\line(0,-1){1}}}
\put(2,0){\rouge{\line(1,0){1}}}
\end{picture}        &
     \begin{picture}(4,4)(0,0)
     \multiput(1.1,0.5)(1,0){1}{\jcarre}
\multiput(0,1)(0,1){3}{\line(1,0){3.5}}
\multiput(1,0)(1,0){3}{\line(0,1){3.5}}
 \put(0,0){\vector(1,0){4}}
 \put(0,0){\vector(0,1){4}}
 \put(0,3){\bleu{\line(1,-1){3}}}
  \linethickness{.5mm}
\put(0,3){\rouge{\line(0,-1){1}}}
\put(0,2){\rouge{\line(1,0){1}}}
\put(1,2){\rouge{\line(0,-1){2}}}
\put(1,0){\rouge{\line(1,0){2}}}
\end{picture}        &
     \begin{picture}(4,4)(0,0)
     \multiput(0.1,1.5)(1,0){1}{\jcarre}
\multiput(0,1)(0,1){3}{\line(1,0){3.5}}
\multiput(1,0)(1,0){3}{\line(0,1){3.5}}
 \put(0,0){\vector(1,0){4}}
 \put(0,0){\vector(0,1){4}}
 \put(0,3){\bleu{\line(1,-1){3}}}
  \linethickness{.5mm}
\put(0,3){\rouge{\line(0,-1){2}}}
\put(0,1){\rouge{\line(1,0){2}}}
\put(2,1){\rouge{\line(0,-1){1}}}
\put(2,0){\rouge{\line(1,0){1}}}
\end{picture}       &
     \begin{picture}(4,4)(0,0)
          \multiput(1.1,0.5)(1,0){1}{\jcarre}
         \multiput(0.1,1.5)(1,0){1}{\jcarre}
 \multiput(0,1)(0,1){3}{\line(1,0){3.5}}
\multiput(1,0)(1,0){3}{\line(0,1){3.5}}
 \put(0,0){\vector(1,0){4}}
 \put(0,0){\vector(0,1){4}}
 \put(0,3){\bleu{\line(1,-1){3}}}
  \linethickness{.5mm}
\put(0,3){\rouge{\line(0,-1){2}}}
\put(0,1){\rouge{\line(1,0){1}}}
\put(1,1){\rouge{\line(0,-1){1}}}
\put(1,0){\rouge{\line(1,0){2}}}
\end{picture}        &
     \begin{picture}(4,4)(0,0)
        \multiput(0.1,0.5)(1,0){2}{\jcarre}
         \multiput(0.1,1.5)(1,0){1}{\jcarre}
\multiput(0,1)(0,1){3}{\line(1,0){3.5}}
\multiput(1,0)(1,0){3}{\line(0,1){3.5}}
 \put(0,0){\vector(1,0){4}}
 \put(0,0){\vector(0,1){4}}
 \put(0,3){\bleu{\line(1,-1){3}}}
  \linethickness{.5mm}
\put(0,3){\rouge{\line(0,-1){3}}}
\put(0,0){\rouge{\line(1,0){3}}}
\end{picture}\\
1 & q & q & q^2 & q^3
       \end{array}$$
       \caption{Dyck paths weighted by $q^{\rm  area}$.}
\label{qcat3}
\end{figure}

\subsection*{Tamari Order}  
The ``{$r$-Tamari poset}'' is obtained by considering the following order on $r$-Dyck paths. 
We say that $a_ia_{i+1}\ldots, a_k$ is a \defn{primitive subsequence}, of a $r$-Dyck path $a_1\,\ldots\,a_n$,  if  
\begin{enumerate}\itemsep=4pt
   \item[(1)] $\bleu{a_j-a_i< r(j-i)}$ for each $i< j\leq k$, and 
   \item[(2)] either $\bleu{k=n}$, or $\bleu{a_{k+1}-a_i\geq r(k+1-i)}$. 
\end{enumerate}
   For each $i$, there is a unique such primitive subsequence. It corresponds to the portion of the $r$-Dyck path that starts at $(a_i,i-1)$ and ends at the ``first return'' of the path to the line of slope $1/r$ passing through the point $(a_i,i-1)$.
Whenever $i$  is such that
$a_{i-1}<a_i$, we set $\bleu{\alpha\leq \beta}$, where 
    \begin{displaymath}\bleu{\beta:=a_1\,\ldots,a_{i-1}\,\rouge{(a_i-1)\,\ldots\,(a_k-1)}\,a_{k+1}\,\ldots\,a_n},\end{displaymath}
with $a_i\,\ldots\,a_k$ equal to the primitive subsequence starting at $a_i$. Recall from~\cite{trivariate} that the  \defn{$r$-Tamari order}\footnote{As far as we know this poset had not yet been considered, before~\cite{trivariate}.} is  the reflexive and transitive closure of this covering relation $\alpha\leq \beta$. Its largest element (denoted, in the usual manner, by $\hat{1}$) is the Dyck path encodes as $00\cdots 0$, and its smallest  (denoted by $\hat{0}$) is the Dyck path for which $a_i=r\,(i-1)$.
For $n=3$, and $r=4$, we get the poset of Figure~\ref{tamari34};
\begin{figure}[ht]
\begin{center}
\scalebox{.6}{\includegraphics{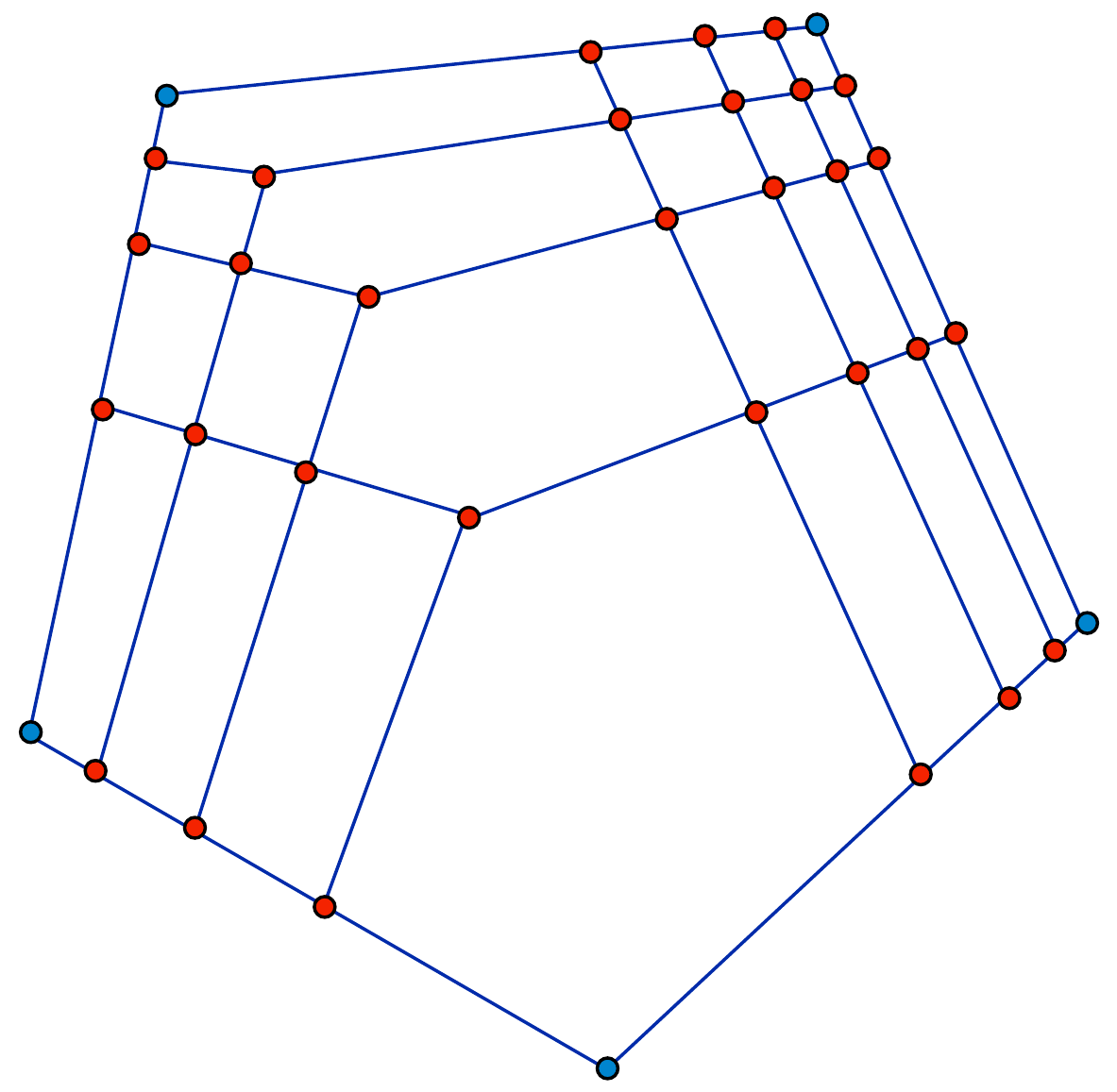}}
\begin{picture}(0,0)(10,0)
\put(2.5,-.51){\bleu{$ 048$}}
                  \put(-1,1.9){\bleu{$\scriptscriptstyle 038$}}
                  \put(-2.7,2.9){\bleu{$\scriptscriptstyle 028$}}
                  \put(-4,3.6){\bleu{$\scriptscriptstyle 018$}}
 \put(9.76,5.8){\bleu{$044$}}
                  \put(7,3.6){\bleu{$\scriptscriptstyle 047$}}
                \put(8.1,4.6){\bleu{$\scriptscriptstyle 046$}}
                  \put(9,5.3){\bleu{$\scriptscriptstyle 045$}}
 \put(-5.8,4.4){\bleu{$008$}}
                 \put(-4.4,8.7){\bleu{$\scriptscriptstyle 007$}}
                                \put(-2.2,7.9){\bleu{$\scriptscriptstyle 017$}}
                                \put(-.75,7.5){\bleu{$\scriptscriptstyle 027$}}
                                \put(1.6,7){\bleu{$\scriptscriptstyle 037$}}
                \put(-4,10.8){\bleu{$\scriptscriptstyle 006$}}
                                 \put(-1.5,10.1){\bleu{$\scriptscriptstyle 016$}}
                                 \put(0,9.8){\bleu{$\scriptscriptstyle 026$}}
               \put(-3.76,12){\bleu{$\scriptscriptstyle 005$}}
                                 \put(-1.1,11.5){\bleu{$\scriptscriptstyle 015$}}
 \put(-4,13){\bleu{$ 004$}}
                 \put(2.3,13.7){\bleu{$\scriptscriptstyle 003$}}
                \put(3.7,13.9){\bleu{$\scriptscriptstyle 002$}}
                 \put(4.75,14){\bleu{$\scriptscriptstyle 001$}}
 \put(6.1,14){\bleu{$000$}}
                  \put(7.9,9.8){\bleu{$\scriptscriptstyle 033$}}
                  \put(6.9,12){\bleu{$\scriptscriptstyle 022$}}
                  \put(6.5,13){\bleu{$\scriptscriptstyle 011$}}
                                   \put(6.55,9){\bleu{$\scriptscriptstyle 034$}}
                                   \put(5.5,11.4){\bleu{$\scriptscriptstyle 023$}}
                                  \put(5,12.45){\bleu{$\scriptscriptstyle 012$}}
                                  \put(5.6,8.7){\bleu{$\scriptscriptstyle 035$}}
                                  \put(4.6,11.2){\bleu{$\scriptscriptstyle 024$}}
                                  \put(4.1,12.3){\bleu{$\scriptscriptstyle 013$}}
                                  \put(4.4,8.2){\bleu{$\scriptscriptstyle 036$}}
                                 \put(3.2,10.8){\bleu{$\scriptscriptstyle 025$}}
                                 \put(2.7,12){\bleu{$\scriptscriptstyle 014$}}
\end{picture}
\end{center}
\caption{The $4$-Tamari poset, for $n=3$.}
\label{tamari34}
\end{figure}
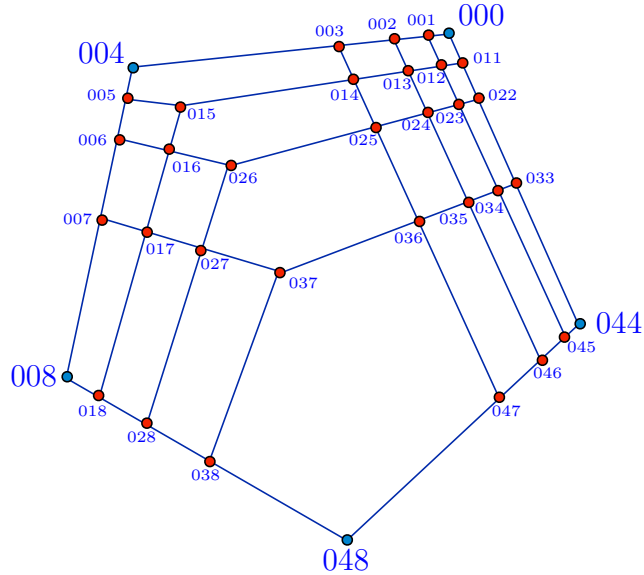
and, for $n=4$ and $r=1$, the poset of Figure~\ref{tamari41}.
 \begin{figure}[ht]
   \setlength{\unitlength}{2.8mm}
\begin{center}
\begin{picture}(0,0)(0,0)
                \put(12,6){$\bleu{{\scriptstyle 0022}}$}  
 \end{picture}
\scalebox{.4}{\includegraphics{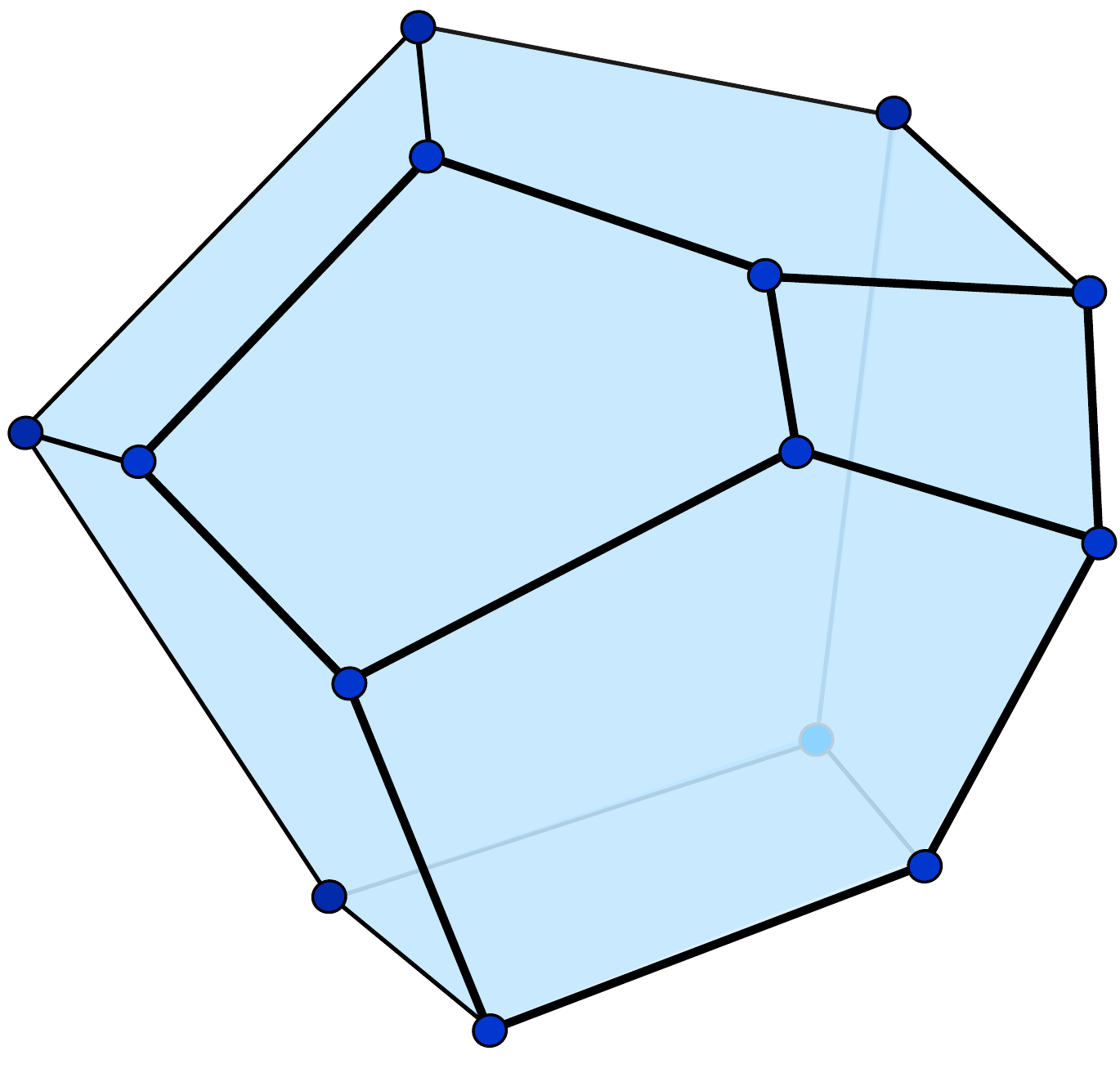}}
\begin{picture}(0,0)(-1,0)
               \put(-17,19){$\bleu{0000}$}
       \put(-24.5,11){$\bleu{0111}$}        \put(-13,16.5){$\bleu{0001}$}      \put(-5,17){$\bleu{0011}$}
  \put(-18,10.3){$\bleu{0112}$}   \put(-11,13){$\bleu{0002}$}  \put(-19,2.5){$\bleu{0122}$}\put(-1.5,14){$\bleu{0012}$}
                 \put(-8,9){$\bleu{0003}$} 
  \put(-14,6){$\bleu{0113}$}      \put(-1,9){$\bleu{0013}$} \put(-6,2){$\bleu{0023}$}
              \put(-16,0){$\bleu{0123}$}
\end{picture}
\end{center}
\vskip-8pt
\caption{The $1$-Tamari poset (associahedron), for $n=4$.}
\label{tamari41}
\end{figure}
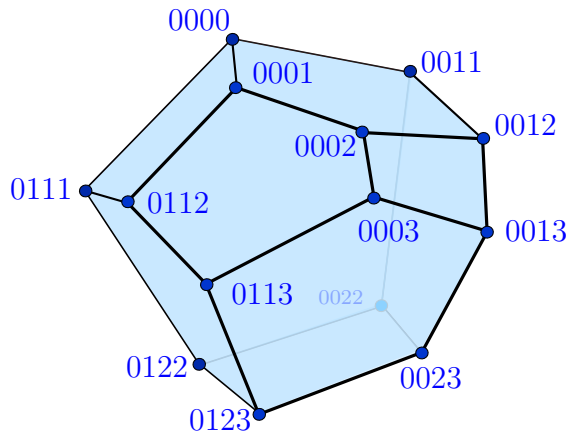
\begin{figure}[ht]
   \setlength{\unitlength}{3mm}
\begin{center}
\scalebox{.3}{\includegraphics{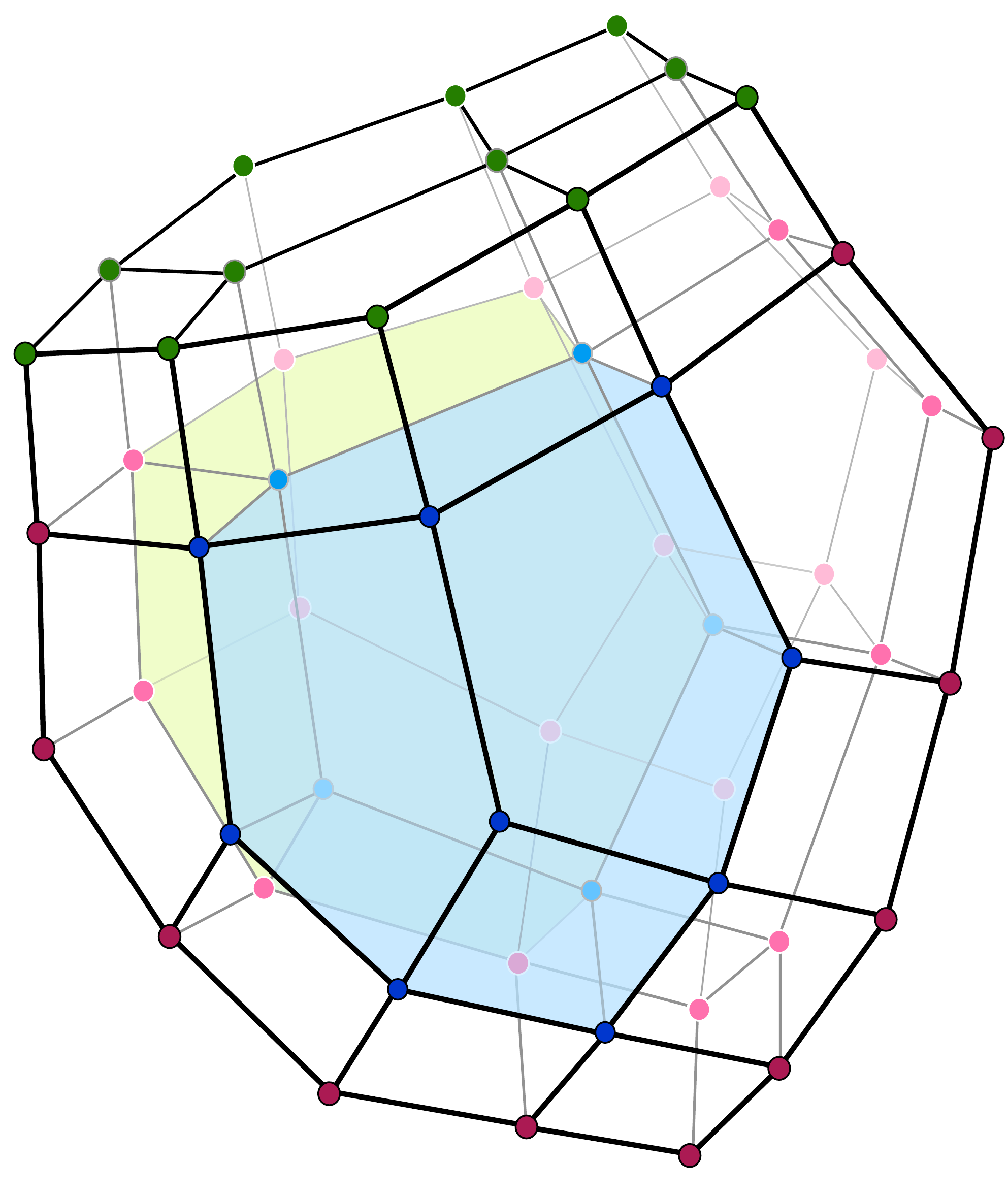}}
\end{center}
\vskip-8pt
\caption{Looking up the $2$-Tamari poset, for $n=4$.}
\label{tamari42}
\end{figure}

It has recently been shown (see \cite{muhle}) that the $r$-Tamari poset is ``EL-shellable''. This has a number of interesting topological consequences.
We also have the geometric realization of Figure~\ref{tamari42} for the poset  $\Dyck{2}{4}$.
This figures suggests that there is a natural polytopal decomposition of  $\Dyck{r}{n}$, with each
polytope a Minkowski sum of polytopal realizations of $1$-Tamari posets (associahedron).
 
\section{Intervals}
As usual, in Poset Theory, an \defn{interval} $[\alpha,\beta]$ is the set of $\gamma\in \Dyck{r}{n}$ such that $\alpha\leq\gamma\leq \beta$. Thus, for $[\alpha,\beta]$ to be non-empty, we must have $\alpha\leq\beta$. Recalling the classical recursive definition of \defn{M\"obius function}  on intervals, we have
\begin{equation}
      \bleu{\mu(\alpha,\beta):=\begin{cases}
      1 & \text{if}\ \alpha=\beta, \\[4pt]
     \displaystyle -\sum_{\alpha\leq \rouge{x}<\beta} \mu(\alpha,\rouge{x})& \text{otherwise}.
\end{cases}}
\end{equation}
In particular, $\mu(\alpha,\beta)=0$ if $\alpha\not\leq \beta$, since the sum is empty in that case. The EL-shellability of the $r$-Tamari poset implies that its M\"obius function takes values in the set $\{-1,0,1\}$. In~\cite{muhle}, this M\"obius function has been calculated for  intervals having homotopy type of spheres, and these are explicitly characterized.

Our aim now is to recall enumeration results for   (non-empty) intervals in $\Dyck{r}{n}$. For example, for $n=3$ and $r=1$, there are $13$ intervals in the corresponding Tamari poset (which appears as the front face of Figure~\ref{tamari41}, with labels obtained by dropping the final $3$):
\begin{displaymath}
            \begin{matrix}
                  [012,012], & [012,011], & [012,002], & [012,001], &  [012,000],\\
                  [002,002], & [002,001], & [002,000],\\
                  [011,011], & [011,000],\\
                  [001,001], & [001,000],\\
                  [000,000].
                  \end{matrix}
                \end{displaymath}
In  \cite{chapoton}, it is shown  that, for $r=1$, the number of intervals is
 $$ \frac {2}{ n\,( n+1) } \binom{4\,n+1}{ n-1}. $$
Computer experiments led the author to conjecture more generally (see~\cite{trivariate}) that $\mathcal{I}_n(r)$,  \defn{the number of intervals in $\Dyck{r}{n}$}, is given by the formula
  \begin{equation}\label{dim_intervals}
      \bleu{\mathcal{I}_n(r)=\frac {( r+1)}{ n\,( r n+1) } \binom{( r+1) ^{2}\,n+r}{ n-1} },
  \end{equation}
  with $\chi(-)$ denoting the function that takes value $1$ or $0$, depending on whether or not its argument is true or false.
  Observe that $\mathcal{I}_n(r)$ is a polynomial in the parameter $r$, indeed
  \begin{eqnarray*}
     \mathcal{I}_1(r)  &=&   1,\\
     \mathcal{I}_2(r)  &=&   (r+1)(r+2)/2,\\
     \mathcal{I}_3(r)  &=&  ( r+1)  ( r+2 )  ( 3\,{r}^{2}+7\,r+3)/6, \\
     \mathcal{I}_4(r)  &=&   ( r+1 )  ( r+2 )  ( 4\,{r}^{2}+9\,r+3)  ( 4\,{r}^{2}+9\,r+4 ) /24,\\
     \mathcal{I}_5(r)  &=&   ( r+1)  ( r+2)  ( 5\,{r}^{2}+11\,r+3)  ( 5\,{r}^{2}+11\,r+4 )  ( 5\,{r}^{2}+11\,r+5)/120.
  \end{eqnarray*}
Formula~\pref{dim_intervals} has been shown to hold by Bousquet-M\'elou  {\sl et al.} in~\cite{melou}.  Their approach exploits the fact that the generating function for the sequence \pref{dim_intervals}, may be described in terms of the series $z(t)$ that satisfies the functional equation
  \begin{equation}
     \bleu{z(t)=t\,(1-z(t))^{-r(r+2)}},
   \end{equation}
 which expands as
    \begin{equation}
     \bleu{z(t)=\sum_{n\geq 1}   \frac{1}{n}\binom{r(r+2)\,n+n-2}{n-1} \, t^n}.
   \end{equation}
 It seems that it would be natural to refine this enumeration of intervals to take account of their homotopy type. Once again the results of~\cite{muhle} strongly suggest that would get nice enumeration formulas in this manner.

\subsection*{Length counting}For $\alpha\leq \beta$ in $r$-Tamari order, let us denote by $d(\alpha,\beta)$ the length of the longest chain going from $\alpha$ to $\beta$, and we say that $d(\alpha,\beta)$ is the \defn{length} of the interval $[\alpha,\beta]$. Observe  that
\begin{lemma}
 \bleu{For all $\alpha\leq \beta$, we have}
    \begin{equation}
         \bleu{d(\alpha,\beta)\leq \area(\beta)-\area(\alpha)}.
     \end{equation}
  \end{lemma}
\begin{proof}[\bf Proof] We need only observe that this is trivially so in the case of the covering relation. \end{proof}
%For at least the two following special situations, we have equality; but this is not so in general.
%\begin{lemma}\label{distance_aire}
% \bleu{For any $\alpha$, we have}
%    \begin{equation}
%         \bleu{d(\hat{0},\alpha)=\area(\alpha)}, \qquad {\rm and}\qquad 
%         \bleu{d(\alpha,\hat{1})=\binom{n}{2}-\area(\alpha')}.
 %    \end{equation}
%  \end{lemma}
%\begin{proof}[\bf Proof] It is easy to see that the Tamari poset is isomorphic to itself, with the reverse order. Hence we need only show the first of these %equality. In view of the preceeding lemma, it suffices to exhibit a chain, from $\hat{0}$ to $\alpha$,  of lenght equal to the area of $\alpha$. We construct %such a chain recursively, by choosing as final step to go from $\beta$ to $\alpha$, where  $\beta$ is obtained from $\alpha$ by adding $1$ to the rightmost %part of $\alpha=a_1\,\ldots\,a_n$ such that $a_i<r(i-1)$.
% \end{proof}

We may refine our interval counting, to take into account this notion of interval length, together with the area. Namely, we consider the polynomial
  $$\bleu{\Dyck{r}{n}(q,t):=\sum_{\alpha\leq \beta} q^{d(\alpha,\beta)}\,t^{d(\beta,\hat{1})}}.$$
Thus, we get
   \begin{eqnarray*}
        \Dyck{1}{1}(q,t)&=&1,\\
        \Dyck{1}{2}(q,t)&=&(1+t)+ q,\\
        \Dyck{1}{3}(q,t)&=&(1+t+2\,{t}^{2}+{t}^{3})+ ( 2+t+2\,{t}^{2} )\, q+ ( 1+t)\, {q}^{2}+{q}^{3},\\
          \Dyck{1}{4}(q,t) &=&(1+t+2\,{t}^{2}+3\,{t}^{3}+3\,{t}^{4}+3\,{t}^{5}+{t}^{6}) \\
          &&\qquad + ( 3+2\,t+4\,{t}^{2}+5\,{t}^{3}+4\,{t}^{4}+3\,{t}^{5} )\, q\\ 
          &&\qquad + ( 3+2\,t+4\,{t}^{2}+3\,{t}^{3}+3\,{t}^{4} )\, {q}^{2}\\ 
          &&\qquad + ( 3+2\,t+2\,{t}^{2}+3\,{t}^{3} )\, {q}^{3}  \\ 
          &&\qquad + ( 2+t+2\,{t}^{2} )\, {q}^{4}+ ( 1+t )\, {q}^{5}+{q}^{6}.
  \end{eqnarray*}
Observe that this polynomial is not symmetric in $q$ and $t$, but we clearly have
    $$\bleu{  \Dyck{r}{n}(q,0)=\Dyck{r}{n}(0,q)}.$$
Another interesting experimental observation, about the distance statistic, is that we seem to have the identity
  $$\bleu{\sum_{\alpha\in \Dyck{r}{n}} q^{d(\rouge{\hat{0}},\alpha)}=\sum_{\alpha\in \Dyck{r}{n}} q^{d(\alpha,\rouge{\hat{1}})}}.$$

\section{Parking functions} For  a sequence of positive integers $\varphi =a_1\,a_2\,\ldots\, a_n$,
let 
\begin{equation}\label{alpha_beta}
   \bleu{\iota(\varphi ):=b_1\, b_2 \cdots  b_n}
    \end{equation}
  be  the \defn{increasing rearrangement} of $\varphi$.
This is to say that $\bleu{b_i\leq b_{i+1}}$, for all $1\leq i<n$. 
For example, for $\varphi =402040101$, we have $ \iota(\varphi )=000011244$.

One says that $\varphi $ is a \defn{$r$-parking function}  if
  $b_1\cdots b_k$ is such that $b_k\leq  r(k-1)$, for all $k$. We denote by \defn{$\Park_n^{(r)}$} the \defn{set of $r$-parking functions}.
    For example, we have the cardinality $49$ set $\Park_3^{(2)}$ containing:
\begin{displaymath}\begin {array}{ccccccc} 
000&001&002&003&004&010&011\\ 
012&013&014&020&021&022&023\\ 
024&030&031&032&040&041&042\\  
100&101&102&103&104&110&120\\  
130&140&200&201&202&203&204\\  
210&220&230&240&300&301&302\\  
310&320&400&401&402&410&420.\end {array}
\end{displaymath}
Perforce, $\iota(\varphi )$ is  a $r$-Dyck path. It is said to be the \defn{shape} of $\varphi $. We  denote by $\bleu{\Park_\alpha}$  the set of all parking functions of shape $\alpha$. 

In view of the above definition, we clearly have a $\S_n$-action (on positions) on the set $\Park_n^{(r)}$ of $r$-parking functions. The sets $\Park_\alpha$  decompose this action into disjoint orbits, and $\S_n$ acts transitively on each $\Park_\alpha$. If $\gamma(\alpha)=c_1c_2\cdots c_k$ is the multiplicity composition of a $r$-dyck path $\alpha$, then the Young subgroup
    $$\bleu{\S_{\gamma(\alpha)}:=\S_{c_1}\!\times \cdots \times \S_{c_k}},$$
     is the \defn{fixator subgroup} of $\alpha$-shaped parking functions. 
It follows that the number of parking functions of shape $\alpha$ is given by the multinomial coefficient
 \begin{equation}\label{dim_parking}
     \bleu{\#\Park_\alpha= \binom{n}{\composition(\alpha)}=\frac{n!}{c_1!c_2!\cdots c_k!}}.
\end{equation}
Parking functions may also be considered as coset representatives of the subgroup $H:=u\,\Z$, with $u=(1,1,\ldots,1)$, of the abelian group $\Z_{r\,n+1}^{n}$. More precisely, one shows that
\begin{lemma}
  \bleu{Each coset of $H$, in  $\Z_{r\,n+1}^{n}$, contains a unique $r$-parking function.}
\end{lemma}
It follows that the set of $r$-parking functions of length $n$,  denoted by \bleu{$\Park_n^{(r)}$},   has cardinality $(rn+1)^{n-1}$.
Thus, we get the  identity
\begin{equation}\label{dim_deux}
   \bleu{(rn+1)^{n-1} = \sum_{\alpha\in \Dyck{r}{n} }  \binom{n}{\composition(\alpha)}}
\end{equation}

\subsection*{More enumeration formulas}
Motivated by an algebraic context discussed in section~\ref{trivariate_diag}, and following computer experiments, the author was led to conjecture (see~\cite{trivariate}) that the following elegant formula holds
  \begin{equation}\label{dim_trois}
      \bleu{ \sum_{\alpha,\beta\in \Dyck{r}{n}} \chi(\alpha\leq\beta)\, \binom{n}{\composition(\beta)}=(r+1)^n\,(r\,n+1)^{n-2} },
  \end{equation}
 In other words, this is the cardinality of the set 
  \begin{equation}\label{ensQ}
      \bleu{\mathcal{Q}_n^{(r)}}:=\bleu{\{(\alpha,\varphi )\ |\ \varphi \in\Park^{(r)}_n\quad{\rm and}\quad \alpha\leq \iota(\varphi )\}}.
    \end{equation}
This was subsequently shown to be the case in~\cite{chapuy}. There approach exploits the fact that the exponential generating function, for the sequence \pref{dim_trois}, may be described as
  \begin{equation}
     \bleu{\sum_{n\geq 0} (r+1)^n\,(rn+1)^{n-2} \frac{t^n}{n!} = (1-r\,z(t))\,e^{(r+1)\,z(t)}},
   \end{equation}
 where $z(t)$ is the series
    \begin{equation}
     \bleu{z(t)=\frac{1}{r(r+1)}\sum_{n\geq 0} n^{n-2}\, \frac{(r(r+1)\,t)^n}{n!} },
   \end{equation}
   which satisfies the functional equation
  \begin{equation}
     \bleu{z(t)=t\,e^{r(r+1)\,z(t)}}.
   \end{equation}

\subsection*{Partitions and Young diagram} Recall that   \defn{(integer) partitions} $\lambda=\lambda_1\lambda_2\cdots \lambda_k$, of $d$, are sequences of integers $\lambda_i\in\N$ such that
 \begin{enumerate}\itemsep=4pt
 \item $\bleu{\lambda_1\geq \lambda_2\geq\cdots \geq\lambda_k>0}$, and
 \item $\bleu{\lambda_1+ \lambda_2+\ldots +\lambda_k=d}$.
 \end{enumerate}
  We write $\bleu{\lambda\vdash d}$, when $\lambda$ is such a {partition of $d$}.
Partitions of integers are often presented in the form of a \defn{Young diagram} (see\footnote{This figure is from  
 \href{http://en.wikipedia.org/wiki/Partition_(number_theory)}{Wikipedia's page on integer partitions}} Figure~\ref{fig_partitions}), for which we typically use the same notation. Recall that this is the following set of cells of $\N\times \N$
    $$\bleu{\lambda=\{(i,j)\ |\ 0\leq i\leq \lambda_{j+1}\}}.$$
\begin{figure}[!ht]
\begin{center}\scalebox{.1}{\includegraphics{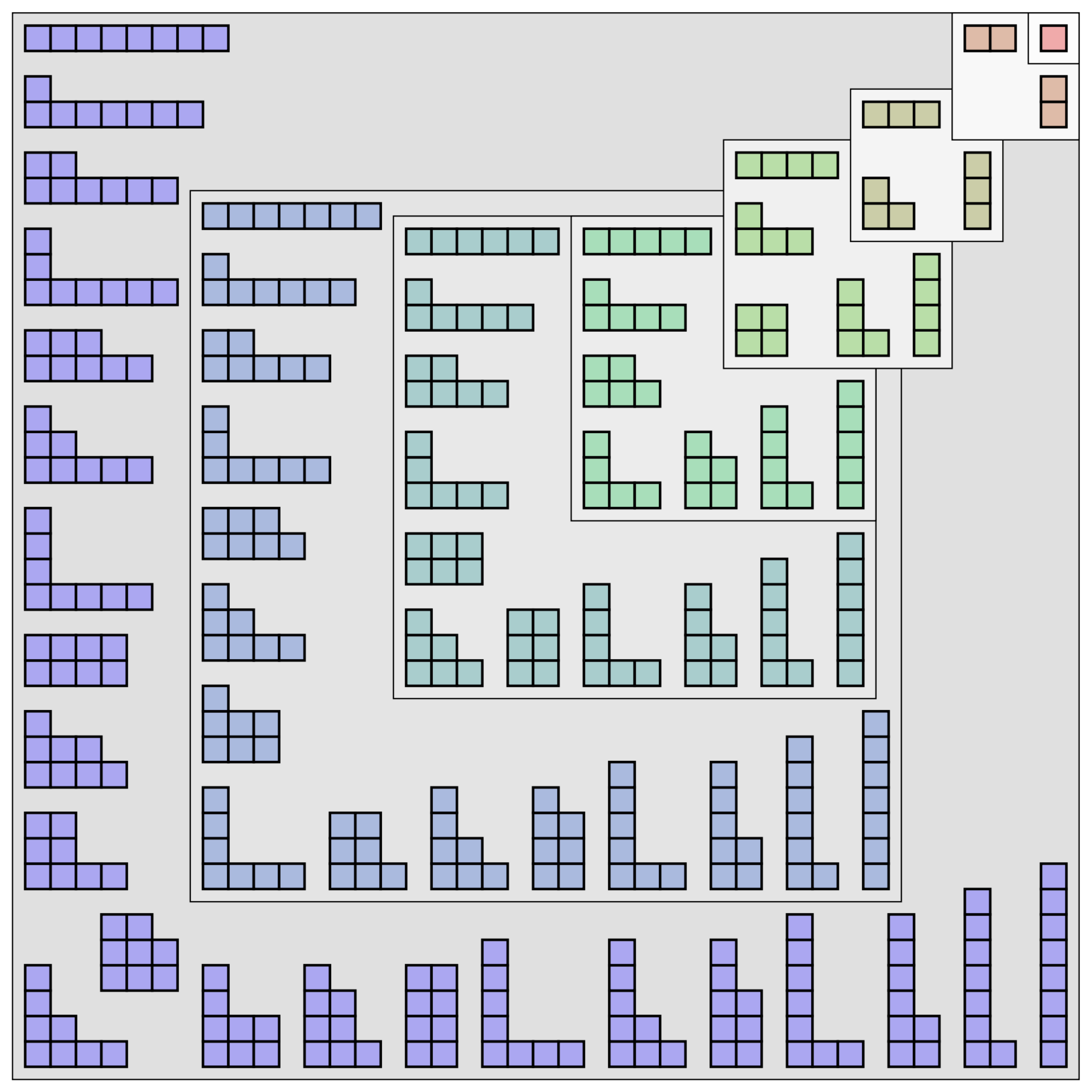}}
 \end{center}\vskip-15pt
\caption{Young diagrams of partitions of $n\leq 8$.}\label{fig_partitions}
\end{figure}

It follows that we have an \defn{inclusion} order on partitions, considering them as subsets of $\N\times \N$. For two partitions $\lambda$ and $\mu$, the \defn{skew diagram} (or \defn{skew partition}) $\lambda/\mu$ is simply the set difference
    $$\bleu{\lambda/\mu:=\{(i,j)\ |\ (i,j)\in\lambda\quad {\rm and}\quad (i,j)\not\in\mu\}}.$$
Clearly, we must have $\mu\subseteq \lambda$, for $\lambda/\mu$ to be non-empty.

We may consider $r$-Dyck paths as the outer boundary of partitions $\lambda$ that are contained in the \defn{$r$-staircase shape} 
    $$\bleu{\delta_n^{(r)}=(r(n-1),r(n-2),\,\ldots\,,r,0)},$$
together with the necessary segments along the axes so that we link $(0,n)$ to $(r n,0)$. Such a path is represented in Figure~\ref{r_park}, with $r=1$, $n=8$, and $\lambda=66531100$. Here we add the necessary number of $0$-parts to turn $\lambda$ into a $n$-part partition (contained in $\delta_n^{(r)}$). This allows us to keep track of $n$, and makes easy the correspondence between such partitions and our previous encoding of $r$-Dyck paths. Indeed, this is but the increasing reordering of the parts of $\lambda$.
\begin{figure}[ht]
\begin{center}
\def\jcarre{\jaune{\linethickness{\unitlength}\line(1,0){1}}}
\def\palecarre{\bleupale{\linethickness{\unitlength}\line(1,0){1}}}
\begin{picture}(9,7.5)(0,.5)
%\multiput(.5,7.5)(0,-1){2}{{\vert{\line(1,-1){6}}}}
\put(0,-.48){\multiput(0,7)(0,1){2}{\jcarre}
                    \multiput(1,5)(0,1){2}{\jcarre}
                    \multiput(3,4)(0,1){1}{\jcarre}
                    \multiput(5,3)(0,1){1}{\jcarre}
                    \multiput(6,1)(0,1){2}{\jcarre}}
\put(0,.45){\multiput(0,0)(0,1){6}{\palecarre}
                    \multiput(1,0)(0,1){4}{\palecarre}
                    \multiput(2,0)(0,1){4}{\palecarre}
                    \multiput(3,0)(0,1){3}{\palecarre}
                    \multiput(4,0)(0,1){3}{\palecarre}
                    \multiput(5,0)(0,1){2}{\palecarre}}
\multiput(0,0)(0,1){9}{\line(1,0){8}}
\multiput(0,0)(1,0){9}{\line(0,1){8}}
  \linethickness{.5mm}
\put(0,8){\rouge{\line(0,-1){2}}\put(.25,-.8){$8$}\put(.25,-1.8){$6$}}
\put(0,6){\rouge{\line(1,0){1}} }
\put(1,6){\rouge{\line(0,-1){2}}\put(.25,-.8){$7$}\put(.25,-1.8){$5$}}
\put(1,4){\rouge{\line(1,0){2}}}
\put(3,4){\rouge{\line(0,-1){1}}\put(.25,-.8){$1$}}
\put(3,3){\rouge{\line(1,0){2}}}
\put(5,3){\rouge{\line(0,-1){1}}\put(.25,-.8){$2$}}
\put(5,2){\rouge{\line(1,0){1}}}
\put(6,2){\rouge{\line(0,-1){2}}\put(.25,-.8){$4$}\put(.25,-1.8){$3$}}
\put(6,0){\rouge{\line(1,0){2}}}
\put(1.5,1.5){\huge $\bleu{\lambda}$}
\end{picture}\end{center}
\caption{The standard tableaux encoding of  $35661010$.}
\label{r_park}
\end{figure}
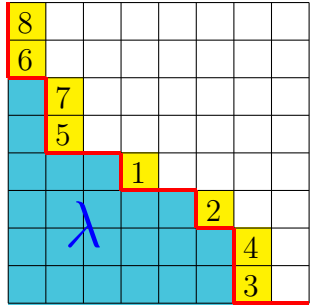

Now, we denote by $\lambda+1^n$ the partition that is obtained by adding one cell in each of the $n$ rows (possibly empty) of $\lambda$. Hence, the skew shape $(\lambda+1^n)/\lambda$ is a vertical strip, i.e.: with exactly one cell in each row. In formula, the corresponding set of cells is (see Figure~\ref{r_park})
    $$\bleu{(\lambda+1^n)/\lambda = \{(i,\lambda_{i+1})\ |\ 0\leq i\leq n-1\}}.$$
Exploiting this point of view on $r$-Dyck paths,  shape $\lambda$ parking functions may simply be considered as \defn{standard tableaux}
   $$\bleu{\tau:(\lambda+1^n)/\lambda\,\bijection\,  \{1,2,\ldots,n\}}.$$
Recall that this means that, for $i<j$, we must have
   $$\bleu{\tau(i,\lambda_{i+1})<\tau(j,\lambda_{j+1})},\qquad {\rm if}\qquad \bleu{\lambda_{i+1}=\lambda_{j+1}}.$$
To get back our previous parking function description from this tableau description, we simply associate to each entry (of the tableau), the number of cells of $\lambda$ that lie (to the left of the entry) on the row in which this entry appears. This process may best be understood by considering an example such as that of Figure~\ref{r_park}.

\subsection*{Standardization} The classical process of \defn{standardization} associates a permutation \bleu{${\rm st}(w)$}  to each length $n$ word $w=v_1\cdots v_n$, on an ordered alphabet. This permutation is obtained by successively replacing the letters in $w$ by the numbers from $1$ to $n$ in the following manner. We first replace all instances of the smallest letter, say there are $k$ of them, by the numbers from $1$ to $k$ going from left to right; then we proceed likewise with the next largest letter, replacing its instances by the numbers $k+1$, $k+2$, etc.; and so on until all letters have been replaced. For example, on the alphabet of the integers,
     $${\rm st}(97750032)= 86751243.$$
Applying this standardization process to parking functions, we get a surjection
    $$\bleu{{\rm st}:\Park_n^{(r)}\rightarrow \S_n},$$
 since permutations appear as special cases of parking functions (up to a shift).
The partition of the set $\Park_n^{(r)}$, whose blocks are the fibers of this surjection, has interesting properties that seem to be tied to  the Shi arrangement. Let us denote by \bleu{$\mathcal{D}^{(r)}_\sigma$} the block associated to a given $\sigma$ in $\S_n$. To see why this notation convention may be considered as ``natural'', observe that the block of the identity corresponds to the set of $r$-Dyck paths. For $n=3$ and $r=2$, the resulting partition of $\Park_3^{(2)}$ consists in the 6 blocks
\begin{eqnarray*}
 \mathcal{D}^{(2)}_{123} &=&\{ 000,001,002,003,004,011,012,013,014,022,023,024 \} , \\
\mathcal{D}^{(2)}_{132} &=&\left\{ 010,020,021,030,031,032,040,041,042 \right\} , \\
\mathcal{D}^{(2)}_{312} &=& \left\{ 100,200,201,300,301,302,400,401,402 \right\} , \\
\mathcal{D}^{(2)}_{213} &=&\left\{ 101,102,103,104,202,203,204 \right\} ,\\
\mathcal{D}^{(2)}_{231} &=& \left\{ 110,120,130,140,220,230,240 \right\} ,\\
\mathcal{D}^{(2)}_{321} &=&\left\{ 210,310,320,410,420 \right\}.
\end{eqnarray*}

\subsection*{Back to the $r$-Tamari order}
In terms of this tableau description, we can extend the $r$-Tamari order to $r$-parking functions. Going up the order consists in moving to the left, and possibly merging, ``correctly chosen'' consecutive columns of  $\tau$. 

The first rule characterizes the set of consecutive columns that may be moved in $\tau$ of shape $(\lambda+1^n)/\lambda$. We first (freely) select a non-empty column of $\tau$, and draw the slope $-1/r$ line that stars a the top left corner of its lowest cell. Going right, this lines must necessarily hit back the $r$-Dyck path $\lambda$ (for the first time). We move one step to the left all columns between (and including) the one selected and the hit point. The second rule is needed only if this hit point happens to be at the bottom of a non-empty column, and merging is need: merging between two consecutive columns being possible only when the entries of the column on the right are all smaller than the entries of the one on the left. Figure~\ref{tamari-park} illustrates all covers, of a given parking function, for this order.
\begin{figure}[ht]
\setlength{\unitlength}{3mm}
\def\jcarre{\jaune{\linethickness{\unitlength}\line(1,0){1}}}
\def\palecarre{\bleupale{\linethickness{\unitlength}\line(1,0){1}}}
\begin{center}
\begin{picture}(9,7.5)(0,.5)
\put(0,-.48){ \multiput(0,5)(0,1){4}{\jcarre}
                    \multiput(1,4)(0,1){1}{\jcarre}
                    \multiput(5,3)(0,1){1}{\jcarre}
                    \multiput(6,1)(0,1){2}{\jcarre}}
\put(0,.45){\multiput(0,0)(0,1){4}{\palecarre}
                    \multiput(1,0)(0,1){3}{\palecarre}
                    \multiput(2,0)(0,1){3}{\palecarre}
                    \multiput(3,0)(0,1){3}{\palecarre}
                    \multiput(4,0)(0,1){3}{\palecarre}
                    \multiput(5,0)(0,1){2}{\palecarre}}
\multiput(0,0)(0,1){9}{\line(1,0){8}}
\multiput(0,0)(1,0){9}{\line(0,1){8}}
  \linethickness{.5mm}
\put(0,8){\rouge{\line(0,-1){4}}\put(.25,-.8){$\scriptstyle 8$}\put(.25,-1.8){$\scriptstyle 7$}
                                               \put(.25,-2.8){$\scriptstyle 6$}\put(.25,-3.8){$\scriptstyle 5$}}
\put(0,4){\rouge{\line(1,0){1}}}
\put(1,4){\rouge{\line(0,-1){1}}\put(.25,-.8){$\scriptstyle 1$}}
\put(1,3){\rouge{\line(1,0){4}}}
\put(5,3){\rouge{\line(0,-1){1}}\put(.25,-.8){$\scriptstyle 2$}}
\put(5,2){\rouge{\line(1,0){1}}}
\put(6,2){\rouge{\line(0,-1){2}}\put(.25,-.8){$\scriptstyle 4$}\put(.25,-1.8){$\scriptstyle 3$}}
\put(6,0){\rouge{\line(1,0){2}}}
\end{picture} \qquad
\begin{picture}(9,7.5)(0,.5)
\put(0,-.48){\multiput(0,7)(0,1){2}{\jcarre}
                    \multiput(1,4)(0,1){3}{\jcarre}
                    \multiput(5,3)(0,1){1}{\jcarre}
                    \multiput(6,1)(0,1){2}{\jcarre}}
\put(0,.45){\multiput(0,0)(0,1){6}{\palecarre}
                    \multiput(1,0)(0,1){3}{\palecarre}
                    \multiput(2,0)(0,1){3}{\palecarre}
                    \multiput(3,0)(0,1){3}{\palecarre}
                    \multiput(4,0)(0,1){3}{\palecarre}
                    \multiput(5,0)(0,1){2}{\palecarre}}
\multiput(0,0)(0,1){9}{\line(1,0){8}}
\multiput(0,0)(1,0){9}{\line(0,1){8}}
  \linethickness{.5mm}
\put(0,8){\rouge{\line(0,-1){2}}\put(.25,-.8){$\scriptstyle 8$}\put(.25,-1.8){$\scriptstyle 7$}}
\put(0,6){\rouge{\line(1,0){1}} }
\put(1,6){\rouge{\line(0,-1){3}}\put(.25,-.8){$\scriptstyle 6$}\put(.25,-1.8){$\scriptstyle 5$}\put(.25,-2.8){$\scriptstyle 1$}}
\put(1,3){\rouge{\line(1,0){4}}}
\put(5,3){\rouge{\line(0,-1){1}}\put(.25,-.8){$\scriptstyle 2$}}
\put(5,2){\rouge{\line(1,0){1}}}
\put(6,2){\rouge{\line(0,-1){2}}\put(.25,-.8){$\scriptstyle 4$}\put(.25,-1.8){$\scriptstyle 3$}}
\put(6,0){\rouge{\line(1,0){2}}}
\end{picture} \qquad
\begin{picture}(9,7.5)(0,.5)
\put(0,-.48){\multiput(0,7)(0,1){2}{\jcarre}
                    \multiput(1,5)(0,1){2}{\jcarre}
                    \multiput(2,4)(0,1){1}{\jcarre}
                    \multiput(4,3)(0,1){1}{\jcarre}
                    \multiput(6,1)(0,1){2}{\jcarre}}
\put(0,.45){\multiput(0,0)(0,1){6}{\palecarre}
                    \multiput(1,0)(0,1){4}{\palecarre}
                    \multiput(2,0)(0,1){3}{\palecarre}
                    \multiput(3,0)(0,1){3}{\palecarre}
                    \multiput(4,0)(0,1){2}{\palecarre}
                    \multiput(5,0)(0,1){2}{\palecarre}}
\multiput(0,0)(0,1){9}{\line(1,0){8}}
\multiput(0,0)(1,0){9}{\line(0,1){8}}
  \linethickness{.5mm}
\put(0,8){\rouge{\line(0,-1){2}}\put(.25,-.8){$\scriptstyle 8$}\put(.25,-1.8){$\scriptstyle 7$}}
\put(0,6){\rouge{\line(1,0){1}} }
\put(1,6){\rouge{\line(0,-1){2}}\put(.25,-.8){$\scriptstyle 6$}\put(.25,-1.8){$\scriptstyle 5$}}
\put(1,4){\rouge{\line(1,0){1}}}
\put(2,4){\rouge{\line(0,-1){1}}\put(.25,-.8){$\scriptstyle 1$}}
\put(2,3){\rouge{\line(1,0){2}}}
\put(4,3){\rouge{\line(0,-1){1}}\put(.25,-.8){$\scriptstyle 2$}}
\put(4,2){\rouge{\line(1,0){2}}}
\put(6,2){\rouge{\line(0,-1){2}}\put(.25,-.8){$\scriptstyle 4$}\put(.25,-1.8){$\scriptstyle 3$}}
\put(6,0){\rouge{\line(1,0){2}}}
\end{picture}
\end{center}
\begin{center}
\begin{picture}(9,12)(0,.5)
\put(0,-.48){\multiput(0,7)(0,1){2}{\jcarre}
                    \multiput(1,5)(0,1){2}{\jcarre}
                    \multiput(2,4)(0,1){1}{\jcarre}
                    \multiput(5,3)(0,1){1}{\jcarre}
                    \multiput(6,1)(0,1){2}{\jcarre}}
\put(0,.45){\multiput(0,0)(0,1){6}{\palecarre}
                    \multiput(1,0)(0,1){4}{\palecarre}
                    \multiput(2,0)(0,1){3}{\palecarre}
                    \multiput(3,0)(0,1){3}{\palecarre}
                    \multiput(4,0)(0,1){3}{\palecarre}
                    \multiput(5,0)(0,1){2}{\palecarre}}
\multiput(0,0)(0,1){9}{\line(1,0){8}}
\multiput(0,0)(1,0){9}{\line(0,1){8}}
  \linethickness{.5mm}
\put(-2,9){\line(-1,1){2}}
\put(4,9){\line(0,1){2}}
\put(10,9){\line(1,1){2}}
\put(0,8){\rouge{\line(0,-1){2}}\put(.25,-.8){$\scriptstyle 8$}\put(.25,-1.8){$\scriptstyle 7$}}
\put(0,6){\rouge{\line(1,0){1}} }
\put(1,6){\rouge{\line(0,-1){2}}\put(.25,-.8){$\scriptstyle 6$}\put(.25,-1.8){$\scriptstyle 5$}}
\put(1,4){\rouge{\line(1,0){1}}}
\put(2,4){\rouge{\line(0,-1){1}}\put(.25,-.8){$\scriptstyle 1$}}
\put(2,3){\rouge{\line(1,0){3}}}
\put(5,3){\rouge{\line(0,-1){1}}\put(.25,-.8){$\scriptstyle 2$}}
\put(5,2){\rouge{\line(1,0){1}}}
\put(6,2){\rouge{\line(0,-1){2}}\put(.25,-.8){$\scriptstyle 4$}\put(.25,-1.8){$\scriptstyle 3$}}
\put(6,0){\rouge{\line(1,0){2}}}
\end{picture}\end{center}
\caption{Covers of $35661100$ (for $r=1$).}
\label{tamari-park}
\end{figure}
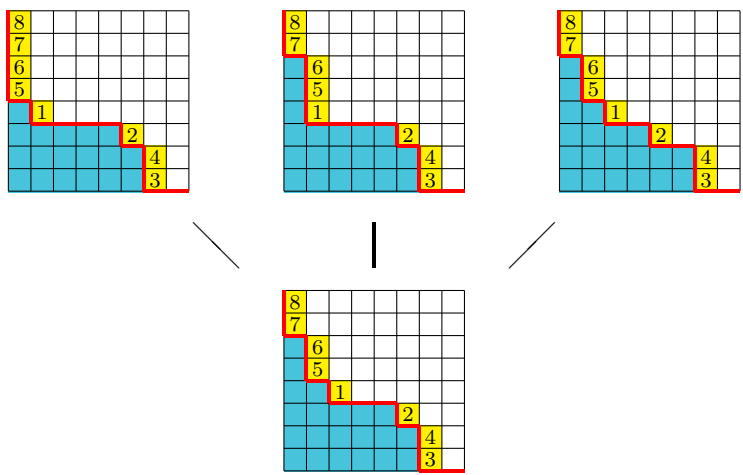

\begin{lemma}
   \bleu{If $\varphi\leq \varphi'$ in $r$-Tamari order,for $r$-parking functions $\varphi$ and $\varphi'$,  then ${\rm st}(\varphi)={\rm st}(\varphi')$. The smallest such parking functions is obtained by filling the shape $(\delta_n^{(r)}+1^n)/\delta_n^{(r)}$ with the permutation $\sigma={\rm st}(\varphi)^{-1}$.}
\end{lemma}

In other words, each of the sets $\mathcal{D}_\sigma^{(r)}$ is equipped with an order, which coincides with the Tamari order when $\sigma$ is the identity permutation.

\subsection*{The $d$-inversion statistic} 
Let us recall, from \cite{HHLRU}, the notion of ``$\dinv$'' statistic on parking functions, which consists in counting their ``{diagonal inversions}''.
For $b>a$, consider two entries $a=\tau(i,j)$ and $b=\tau(i',j')$, in such a standard tableau $\tau$, of shape $(\lambda+1^n)/\lambda$, respectively sitting in the cells $(i,j)$ and $(i',j')$. The pair $(b,a)$ is said to be a \defn{diagonal inversion} if either
\begin{enumerate}
   \item[(i)] $a$ and $b$ lie on the same diagonal, with $a$ to the right of $b$. Equivalently, this to say that
        $$\bleu{i'+j'=i+j} \qquad{\rm and}\qquad \bleu{i>i'},$$
   \item[(ii)] or, $b$ lies on the ``next'' diagonal from that of $a$, with $a$ to the left of $b$. Equivalently, this to say that
        $$\bleu{i'+j'=(i+j)\rouge{+1}} \qquad{\rm and}\qquad \bleu{i<i'}.$$
\end{enumerate}
In Figure~\ref{r_park}, the pairs $(8,2),(8,4),(6,1), (6,3), (7,1)$, and $(7,3)$
are diagonal inversions of the first type; whereas the pairs $(4,1)$ and  $(2,1)$
are diagonal inversions of the second type.
Observe that this notion of diagonal inversion also makes sense for semi-standard tableaux of shape $(\lambda+1^n)/\lambda$.

We now extend the notion of diagonal-inversion to the context of $r$-parking functions. We first observe that, for a fixed $r$, two $1\times 1$ cells (in $\N\times \N$)
intersect several lines of slope $-1/r$. Indeed, consider two such cells with south-west corner sitting respectively in position $(i,j)$ and $(i',j')$. The number of consecutive diagonals, with equation of the form $x+r\,y=k$ ($k\in\N$), that cross the ``west wall'' of the two cells, is equal to 
             $$\bleu{\max(0,r-|k-k'|)},$$
 where $\bleu{k=i+rj}$ and $\bleu{k'=i'+rj'}$. In this expression, $r$ is the number of such diagonals that cross the west wall of the cell containing $b$; whereas (in the case where it makes sense) $|k-k'|$ is the number of such diagonals that only cross the west wall of the cell containing $b$.
Hence the difference counts the right number of diagonals, as illustrated in Figure~\ref{figcross}.
 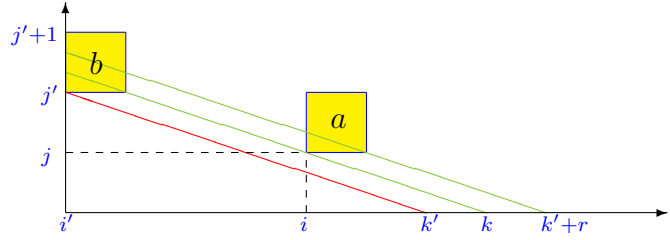
\begin{figure}[ht]
\begin{center}\setlength{\unitlength}{4mm}
\begin{picture}(10,7)(5,0)
\newsavebox{\carre}
\savebox{\carre}(2,2)[1]{\put(0,1){\jaune{\linethickness{8mm}\line(1,0){2}}}\multiput(0,0)(2,0){2}{\bleu{\line(0,1){2}}}
                                     \multiput(0,0)(0,2){2}{\bleu{\line(1,0){2}}}}
% \thicklines
 \put(0,4){\usebox{\carre}}
  \put(-1.8,5.7){\bleu{$\scriptstyle j'+1$}}
 \put(-.8,3.7){\bleu{$\scriptstyle j'$}}
  \put(-.8,1.7){\bleu{$\scriptstyle j$}}
\put(-0.2,-.6){\bleu{$\scriptstyle i'$}}
  \put(7.8,-.6){\bleu{$\scriptstyle i$}}
    \put(11.8,-.6){\bleu{$\scriptstyle k'$}}
    \put(13.8,-.6){\bleu{$\scriptstyle k$}}
        \put(15.8,-.6){\bleu{$\scriptstyle k'+r$}}
      \put(8,2){\usebox{\carre}}
% \put(17,0){\vert{\line(-3,1){17}}}
   \put(16,0){\vert{\line(-3,1){16}}}
  % \put(15,0){\vert{\line(-3,1){15}}}
   \put(14,0){\vert{\line(-3,1){14}}}
   %  \put(13,0){\rouge{\line(-3,1){13}}}
       \put(12,0){\rouge{\line(-3,1){12}}}
 \thinlines
 \put(0,0){\vector(1,0){20}}
  \put(0,0){\vector(0,1){7}}
\multiput(0,2)(.5,0){16}{\line(1,0){.25}}
\multiput(8,0)(0,.5){4}{\line(0,1){.25}}
\put(.8,4.6){$b$}
\put(8.8,2.8){$a$}
  \end{picture}\end{center}
  \caption{Diagonals of slope $-1/r$ crossing two cells of $(\lambda+1^n)/\lambda$.}\label{figcross}
  \end{figure}

The \bleu{$r$-diagonal inversions} occur on such diagonals, as well as on ``near'' diagonals. Technically, consider a $r$-Dyck path as a partition $\lambda$ contained in the $r$-staircase  $\delta_n^{(r)}$. For $a<b$, let $a=\tau(i,j)$ and $b=\tau(i',j')$ appear as entries in a (semi-)standard tableau $\tau$ of shape $(\lambda+1^n)/\lambda$. The pair $(b,a)$ counts as a $r$-diagonal inversion, as many times as it occurs along a line $x+r\,y=k$ crossing both west-walls of the cells in which they appear, when $i>i'$; and as many times
as a similar property occurs for near diagonals, when $i<i'$. In formula, set $k=i+rj$ and $k'=i'+rj'$, then we have
\begin{enumerate}\itemsep=6pt
   \item[(i)]  $\bleu{\max(0,r-|k'-k|)}$ $r$-diagonal inversions, if $\bleu{i>i'}$, and
   \item[(ii)] $\bleu{\max(0,r-|k'-k-1|)}$ $r$-diagonal inversions, if $\bleu{i<i'}$.
\end{enumerate}
To extend this to the semi-standard case, we consider that $a=b$ contributes for as many   $r$-diagonal inversions, as the minimum of the two values
    $$\bleu{\max(0,r-|k'-k|)} \qquad {\rm and}\qquad \bleu{\max(0,r-|k'-k-1|)}.$$
We denote by \bleu{$\dinv_{r}(\tau)$} the \defn{number of $r$-diagonal inversions} of a semi-standard tableau $\tau$.

In the correspondence between standard tableaux  and parking function, let us replace standard tableaux by semi-standard tableaux. The result is said to be a \defn{$r$-semi-parking function}, and we denote by \bleu{$\mathcal{W}_n^{(r)}$} and the set  $r$-semi-parking functions of height $n$. Similarly, we denote by \bleu{$\mathcal{W}_\alpha$} the set of $r$-semi-parking functions that are of shape $\alpha$, for a given $r$-Dyck path $\alpha$. Observe that both these sets are infinite, since there is no bound on the possible values of the entries. 

\subsection*{Diagonal descents} We now aim at understanding how certain inequalities between entries of a semi-parking function are sufficient to characterize all positions where diagonal-inversions occur in this tableau. This is the reason behind the introduction of the ``{diagonal descent set}'' of a parking function.
\begin{figure}[ht]
\begin{center}
\setlength{\unitlength}{4mm}
\def\jcarre{\jaune{\linethickness{\unitlength}\line(1,0){1}}}
\def\palecarre{\bleupale{\linethickness{\unitlength}\line(1,0){1}}}
\begin{picture}(17,7.5)(0,.5)
%\multiput(.5,7.5)(0,-1){2}{{\vert{\line(1,-1){6}}}}
\put(0,-.48){\multiput(0,7)(0,1){2}{\jcarre}
                    \multiput(3,5)(0,1){2}{\jcarre}
                    \multiput(6,4)(0,1){1}{\jcarre}
                    \multiput(10,3)(0,1){1}{\jcarre}
                    \multiput(12,1)(0,1){2}{\jcarre}}
\put(0,.45){\multiput(0,0)(1,0){12}{\palecarre}
                    \multiput(0,1)(1,0){12}{\palecarre}
                    \multiput(0,2)(1,0){10}{\palecarre}
                    \multiput(0,3)(1,0){6}{\palecarre}
                    \multiput(0,4)(1,0){3}{\palecarre}
                    \multiput(0,5)(1,0){3}{\palecarre}}
\multiput(0,0)(0,1){9}{\line(1,0){16}}
\multiput(0,0)(1,0){17}{\line(0,1){8}}
  \linethickness{.5mm}
\put(0,8){\rouge{\line(0,-1){2}}\put(.25,-.8){$8$}\put(.25,-1.8){$6$}}
\put(0,6){\rouge{\line(1,0){3}} }
\put(3,6){\rouge{\line(0,-1){2}}\put(.25,-.8){$7$}\put(.25,-1.8){$5$}}
\put(3,4){\rouge{\line(1,0){3}}}
\put(6,4){\rouge{\line(0,-1){1}}\put(.25,-.8){$1$}}
\put(6,3){\rouge{\line(1,0){4}}}
\put(10,3){\rouge{\line(0,-1){1}}\put(.25,-.8){$2$}}
\put(10,2){\rouge{\line(1,0){2}}}
\put(12,2){\rouge{\line(0,-1){2}}\put(.25,-.8){$4$}\put(.25,-1.8){$3$}}
\put(12,0){\rouge{\line(1,0){4}}}
\end{picture}\end{center}
\caption{Diagonal reading permutation $82476135$, for $r=2$.}
\label{r_park_diag}
\end{figure}
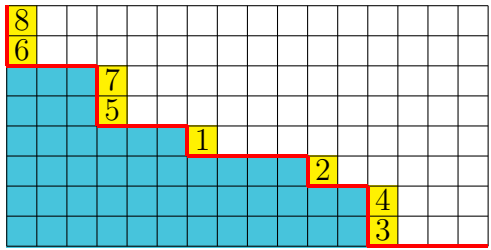
Consider the \defn{reverse $r$-diagonal lexicographic order} on $\N\times \N$ defined by
    $$\bleu{(i,j)<_d (i',j')\qquad{\rm iff}\qquad \begin{cases}
      i+rj>i'+rj' & \text{or}, \\[4pt]
      i+rj=i'+rj',& {\rm and}\ i<i'.
\end{cases}}$$
Reading the entries of a semi-parking function $\varphi$ in increasing $<_d$ order, we obtain its \defn{$r$-diagonal reading word} $w_r(\varphi)$. Informally, we read each diagonal from left to right, starting with the one that is furthest from the origin.  Clearly, when $\varphi$ is a  parking function, this reading word is a permutation of $n$. For example, $82476135$ is the $r$-diagonal reading permutation of the parking function of Figure~\ref{r_park_diag}.

The \defn{$r$-diagonal descent set} of a $r$-parking function $\varphi$, denoted by \bleu{$\Desc_r(\varphi)$},  is the set of entries $a$ of $\varphi$,  for which $a+1$ appears in a higher position than $a$ in $<_d$ order. In other words, $a+1$ appears to the right of $a$ in the $r$-diagonal reading word of $\varphi$. For example, $\{1,3,5,6,7\}$ is the descent set of the $2$-parking function of Figure~\ref{r_park_diag}.

 Recall that the \defn{descent set} of a permutation $\sigma$ is
    $$\bleu{D(\sigma):=\{i \ |1\leq i<n,\ \sigma_i>\sigma_{i+1}\}}.$$
The $r$-diagonal descent set of a semi-parking function $\varphi$ may be described in terms of this classical notion as:
   $$\bleu{\Desc_r(\varphi)=D({\rm st}(w_r(\varphi))^{-1})}.$$
 This is simply the descent set of the inverse of ${\rm st}(w_r(\varphi))$.

\subsection*{Array encoding} To better manipulate all these notions, it is useful to encode height $n$ semi-parking functions as
$2\times n$ arrays, in the following manner. Let us label diagonals of the augmented $r$-staircase shape $\delta_n^{(r)}+1^n$,  by their distance from the outermost $r$-diagonal. Thus a cell $(i,j)$ belongs to the $k^{\rm th}$-diagonal if and only if
     $$\bleu{r(n-1-j)-i=k}.$$
 \begin{figure}[ht]
\begin{center}
\begin{picture}(13,5)(0,.5)
\multiput(0,0)(0,1){6}{\line(1,0){13}}
\multiput(0,0)(1,0){14}{\line(0,1){5}}
 \linethickness{.5mm}
\multiput(0,5)(3,-1){4}{\rouge{\line(0,-1){1}\put(0,-1){\line(1,0){3}}}}
\put(12,1){\rouge{\line(0,-1){1}\put(0,-1){\line(1,0){1}}}}
\put(0.3,0.2){\bleu{
           \put(0,4){$0$}
            \put(0,3){\put(0,0){$3$}\put(1,0){$2$}\put(2,0){$1$}\put(3,0){$0$}}
           \put(0,2){\put(0,0){$6$}\put(1,0){$5$}\put(2,0){$4$}\put(3,0){$3$}\put(4,0){$2$}\put(5,0){$1$}\put(6,0){$0$}}
            \put(0,1){\put(0,0){$9$}\put(1,0){$8$}\put(2,0){$7$}\put(3,0){$6$}\put(4,0){$5$}\put(5,0){$4$}\put(6,0){$3$}\put(7,0){$2$}\put(8,0){$1$}\put(9,0){$0$}}
              \put(0,0){\put(-0.2,0){\put(0,0){$12$}\put(1,0){$11$}\put(2,0){$10$}}\put(3,0){$9$}\put(4,0){$8$}\put(5,0){$7$}\put(6,0){$6$}\put(7,0){$5$}\put(8,0){$4$}\put(9,0){$3$}\put(10,0){$2$}\put(11,0){$1$}\put(12,0){$0$}}
}}
 \end{picture}\end{center}
\caption{Diagonal labeling, for $n=5$ and $r=3$.}
\label{r_diag_label}
\end{figure}
Such a labeling is illustrated in Figure~\ref{r_diag_label}.

We construct a vector $(v_1,v_2,\ldots, v_n)$ by  listing the entries of a semi-parking function, reading down columns starting with the leftmost, and going from left to right.  In the same order, we record in a word $(u_1,u_2,\ldots, u_n)$ the diagonals to which each entry belongs. The array encoding of a semi-parking function $\varphi$ then consists in 
  $$\bleu{\varphi=\begin{bmatrix}
         v_1&v_2&\cdots & v_n\\
         u_1&u_2&\cdots & u_n \end{bmatrix}}.$$
 Since we can clearly recuperate its original description from this data, we use the same notation both for the semi-parking function and its array encoding.
Observe that the $r$-semi-parking function condition translates to the following two constraints on this array. For $1<i\leq n$, we have
\begin{enumerate}\itemsep=6pt
    \item[(i)]   $\bleu{0\leq u_i\leq u_{i-1}+r}$, and
   \item[(ii)]   if \bleu{ $u_i=u_{i-1}+r$} then \bleu{$v_i>v_{i-1}$}. 
\end{enumerate}
For example, the encoding that corresponds to the $2$-parking function of Figure~\ref{r_park_diag}, is
  $$\begin{bmatrix}
         8 & 6 & 7 & 5 & 1 & 2 & 4 & 3\\
         0 & 2 & 1 & 3 & 2 & 0 & 0 & 2 \end{bmatrix}.$$
Two consecutive entries $v_{i-1}$  and $v_i$ sit in the same column, exactly when $u_i=u_{i-1}+r$. This shows why we must have condition (ii) above.
Also, the $r$-diagonal reading word \bleu{$w_r(\varphi)$} of $\varphi$ is easily obtained by successive left to right readings of the top row, according to increasing values of the bottom one.
In our running example, we first read off $8,2,4$ that sit above $0$'s; then $7$, sitting above a $1$; then $6,3$, sitting above the $2$'s; and finally $5$ which sits above $3$. 
Our previous area statistic is simply the sum of the $u_i$:
    $$\bleu{\area(\varphi)=\sum_{i=1}^n u_i}.$$

 \section{Polya/Frobenius enumeration}
It is interesting to take into account the action of $\S_n$ on $r$-parking function of height $n$. 
The typical tool for this is the {Polya cycle index series}, or equivalently the {Frobenius characteristic}. The first notion is usually considered in the context of permutation actions, but it can be viewed as a special case of the Frobenius transform of the character a linear representation of the symmetric group. To ease future considerations, we will adopt this second more general approach. 

\subsection*{Symmetric functions}
To simplify the presentation, we assume that  $\mbf{w}=w_1,w_2,w_3,\ldots$
is a denumerable set of variables. The ring $\Lambda$ of symmetric polynomials\footnote{We will often say functions to underline that we work with ``polynomials'' in infinitely many variables.} in the variables $\mbf{w}$ is graded by degree:
    \begin{displaymath}
        \bleu{\Lambda=\bigoplus_d\, \Lambda_d},
    \end{displaymath}   
 where $ \Lambda_d$ denotes the degree $d$ homogeneous component. It is easy to see that $ \Lambda_d$ affords as basis the set $\{m_\lambda(\mbf{w})\}_\lambda$ of \defn{monomial} symmetric polynomials. Recall that these are indexed by  partitions $\lambda=\lambda_1\lambda_2\cdots \lambda_k$ of $d$. Then 
  $$
    \bleu{m_\lambda(\mbf{w}) = \sum_{i_1,\ldots,i_k} w_{i_1}^{\lambda_1} w_{i_2}^{\lambda_2}\cdots w_{i_k}^{\lambda_k}},
  $$ 
  where the sum is over all possible choices of $k$ distinct indices. For example, we have
  $$
     m_{21}(x,y,z,\ldots)= x^2y+x^2z+y^2z+x\,y^2+x\,z^2+y\,z^2+\ldots
  $$
We mostly use the notation of~\cite{macdonald}, so that $h_k(\mbf{w})$, $e_k(\mbf{w})$, $p_k(\mbf{w})$ respectively  denote the \defn{complete homogeneous}, \defn{elementary}, and \defn{power sum} symmetric polynomials. These are characterized by the following generating function identities
\begin{eqnarray}
     \bleu{\sum_{k\geq 0}h_k(\mbf{w})\,t^k}&=&\bleu{\prod_{i} \frac{1}{1-w_i\,t}},\label{genh} \\
     \bleu{ \sum_{k\geq 0}e_k(\mbf{w})\,t^k}&=&\bleu{ \prod_{i} {(1+w_i\,t)}} ,\qquad {\rm and} \label{gene}\\
     \bleu{ \sum_{k\geq 1}\frac{p_k(\mbf{w})\,t^k}{k} }&=&\bleu{-\sum_i \log(1-w_i\,t)} .\label{genp}
  \end{eqnarray}
  Recall that, setting 
 \begin{eqnarray*}
       \bleu{h_\lambda(\mbf{w})}&:=&\bleu{h_{\lambda_1}(\mbf{w})\cdots h_{\lambda_k}(\mbf{w})},\qquad
        \bleu{e_\lambda(\mbf{w})}:=\bleu{e_{\lambda_1}(\mbf{w})\cdots e_{\lambda_k}(\mbf{w})},\qquad {\rm and} \\
        \bleu{p_\lambda(\mbf{w})}&:=&\bleu{p_{\lambda_1}(\mbf{w})\cdots p_{\lambda_k}(\mbf{w})},
      \end{eqnarray*}
   the three sets $\{h_\lambda(\mbf{w})\}_{\lambda\vdash d}$, $\{e_\lambda(\mbf{w})\}_{\lambda\vdash d}$, and $\{p_\lambda(\mbf{w})\}_{\lambda\vdash d}$ are bases of $\Lambda_d$.
The classical \defn{Cauchy kernel identity} (see Macdonald's book ~\cite{macdonald}) states that 
     \begin{equation}\label{Cauchy}
         \bleu{\prod_{i,j}\frac{1}{1-w_i,v_j}= \sum_{\mu} f_\mu(\mbf{w})\, g_\mu(\mbf{v})},
      \end{equation}
holds, if and only if $\{f_\mu(\mbf{w})\}_\mu$ and  $\{g_\mu(\mbf{w})\}_\mu$ form a \defn{dual pair} of bases for the \defn{Hall scalar product}, that is 
   $$\langle f_\lambda(\mbf{w}),g_\mu(\mbf{w})\rangle = \begin{cases}
     1 & \text{if }\ \lambda=\mu, \\
     0 & \text{otherwise}.
\end{cases}$$ 
The two basis $\{p_\lambda(\mbf{w})\}_\lambda$ and $\{{p_\lambda(\mbf{w})}/{z_\lambda}\}_\lambda$ constitute such a dual pair, if $z_\lambda$ denotes the integer
   $$\bleu{z_\lambda:=1^{d_1}\,d_1! 2^{d_2}\,d_2! \cdots n^{d_n}\,d_n! },$$
   with $d_i=d_i(\lambda)$ equal to the number of parts of size $i$ in $\lambda$. The monomial and complete homogeneous basis are dual to one another. 
The Schur functions $s_\lambda(\mbf{w})$ are characterized by the fact that they form an orthonormal basis 
  $$\langle s_\lambda(\mbf{w}),s_\mu(\mbf{w})\rangle = \begin{cases}
     1 & \text{if }\ \lambda=\mu, \\
     0 & \text{otherwise},
\end{cases}$$
and that they expand as
\begin{equation}\label{stom}
    \bleu{s_\lambda(\mbf{w})=m_\lambda(\mbf{w})+\sum_{\mu} K_{\lambda,\mu}\, m_\mu(\mbf{w})},
  \end{equation}
 with coefficients $K_{\lambda,\mu}$ known as the \defn{Kostka numbers}.
 Recall that these are the positive integers obtained by counting the number of semi-standard tableaux of shape $\lambda$, and \defn{content} $\mu$.
 This is to say, tableaux that are filled with $\mu_i$ copies of the integer $i$. In particular, the partition $\mu$ has to be smaller than $\lambda$ in \defn{dominance} order, so that
 \begin{eqnarray*}
     &&\mu_1\geq \lambda_1,\\
     &&\mu_1+\mu_2\geq \lambda_1+ \lambda_2,\\
     &&\mu_1+\mu_2+\mu_3\geq \lambda_1+ \lambda_2+ \lambda_3,\\
     &&{\rm etc.}
 \end{eqnarray*}
Indeed, if $j>i$, the integer $i$ cannot appear in row $j$; thus all copies of integers smaller or equal to $i$ must placed in the first $i$ rows. 
For example, we have
\begin{eqnarray*}
&& s_{{4}}(\mbf{w})=m_{{4}}(\mbf{w})+m_{{31}}(\mbf{w})+m_{{22}}(\mbf{w})+m_{{211}}(\mbf{w})+m_{{1111}}(\mbf{w})\\
&&s_{{31}}(\mbf{w})=m_{{31}}(\mbf{w})+m_{{22}}(\mbf{w})+2\,m_{{211}}(\mbf{w})+3\,m_{{1111}}(\mbf{w})\\
&&s_{{22}}(\mbf{w})=m_{{22}}(\mbf{w})+m_{{211}}(\mbf{w})+2\,m_{{1111}}(\mbf{w})\\
&&s_{{211}}(\mbf{w})=m_{{211}}(\mbf{w})+3\,m_{{1111}}(\mbf{w})\\
&&s_{{1111}}(\mbf{w})=m_{{1111}}(\mbf{w})
\end{eqnarray*}
Observe that $s_n(\mbf{w})=h_n(\mbf{w})$, and $s_{11\cdots1}(\mbf{w})=e_n(\mbf{w})$.
As usual $\omega$ stands for the involution that sends the Schur function $s_\mu(\mbf{w})$ to the Schur function $s_{\mu'}(\mbf{w})$, indexed by the conjugate partition. 
Thus $\omega h_n(\mbf{w})=e_n(\mbf{w})$. It may also be worth 
recalling that $\omega\,p_j(\mbf{w})= (-1)^{j-1}p_j(\mbf{w})$.

\subsection*{Frobenius characteristic} Given a linear action of $\S_n$ on a $d$-dimensional vector space\footnote{In our case, this could be a free vector space, on some set $E$ on which $\S_n$ is acting by permutation.} $\mathcal{V}$, the associate \defn{Frobenius characteristic} is defined as
   \begin{equation}\label{def_frobsimple}
       \bleu{\mathcal{V}(\mbf{w}):=\frac{1}{n!}\sum_{\sigma\in \S_n} {\rm Trace}(M_\sigma)\, p_{\lambda(\sigma)}(\mbf{w})}, 
   \end{equation}
where $M_\sigma$ is the $d\times d$ matrix corresponding to the action of $\sigma$ on $\mathcal{V}$, and $\lambda(\sigma)$ is the partition corresponding\footnote{Its parts are the cycle lengths.} to the cycle structure of $\sigma$. Observe that, when $\S_n$ acts by permutation on a set $E$ (letting $\mathcal{V}:=\C\,E$ be the corresponding free vector space), then the trace of the matrix $M_\sigma$ is simply the number of fixed points of $\sigma$ acting on $E$.   
Observe also that $\dim(\mathcal{V})$ appears in $\mathcal{V}(\mbf{w})$ as the coefficient of $p_1^n/n!$.

The Frobenius characteristic establishes a correspondence between linear representations of $\S_n$, and symmetric functions. This correspondence is linear, and sends irreducible representations to Schur functions. More specifically, we have
\begin{eqnarray}
    &&\bleu{(\mathcal{V}_1\oplus\mathcal{V}_2) (\mbf{w})}=\bleu{\mathcal{V}_1(\mbf{w})+\mathcal{V}_2(\mbf{w})},\\
    &&\bleu{(\mathcal{V}_1\otimes \mathcal{V}_2) (\mbf{w})}=\bleu{\mathcal{V}_2(\mbf{w})\cdot \mathcal{V}_2(\mbf{w})},\label{tensor}\\
    &&\bleu{\mathcal{V}(\mbf{w})}=\bleu{s_\lambda(\mbf{w})}\qquad {\rm iff}\qquad \mathcal{V}\quad\hbox{is \defn{irreducible}},\\
     &&\bleu{\mathcal{V}(\mbf{w})}=\bleu{s_{n}(\mbf{w})}\qquad {\rm iff}\qquad \mathcal{V}\quad\hbox{is the \defn{trivial} representation},\label{trivial}\\
     &&\bleu{\mathcal{V}(\mbf{w})}=\bleu{s_{11\cdots 1}(\mbf{w})}\qquad {\rm iff}\qquad \mathcal{V}\quad\hbox{is the \defn{sign} representation},\\
     &&\bleu{\dim\,\mathcal{V}}=\bleu{\frac{1}{n!}\,\langle\mathcal{V}(\mbf{w}),p_1^n \rangle}.
 \end{eqnarray}
Recall that there are exactly as many irreducible representations of $\S_n$, as there are conjugacy classes in $\S_n$, and that these are naturally in bijections with partitions.
Indeed, a conjugacy class is entirely characterized by the cycle structure of permutations that lie in it. 

Let $\{\mathcal{V}_\lambda\}_{\lambda\vdash n}$ be a complete set of representatives of irreducible representations of $\S_n$. The fundamental theorem of representation theory, stating that any representation decomposes uniquely into irreductibles, may thus be casted equivalently as
 \begin{equation}\label{expansion_schur}
   \bleu{\mathcal{V}=\bigoplus_{\lambda\vdash n} a_\lambda\, \mathcal{V}_\lambda}\qquad {\rm iff}\qquad 
   \bleu{\mathcal{V}(\mbf{w})= \sum_{\lambda\vdash n} a_\lambda\, s_\lambda(\mbf{w})},
\end{equation}
where the positive integers $a_\lambda$ give the multiplicity of the irreducible representation $\mathcal{V}_\lambda$ in $\mathcal{V}$.
In other words, the symmetric function $\mathcal{V}(\mbf{w})$ contains all the necessary information about the representation $\mathcal{V}$: its dimension,
how it decomposes into irreducible, etc; in an easy to retrieve fashion. 

Recall that $\dim(\mathcal{V}_\lambda)$ is the number of standard tableaux of shape $\lambda$, given by the well known \defn{hook formula}
\begin{equation}
   \bleu{f_\lambda=\frac{n!}{\prod_{(i,j)\in \lambda} h_{ij}}},
\end{equation}
where we attach to each cell $(i,j)$, in $\lambda$, its \defn{hook length} 
   $$\bleu{h_{ij}:= \lambda_{j+1}+\lambda_{i+1}'-i-j-1}.$$
Thus, when we have \pref{expansion_schur},
$$\bleu{\dim(\mathcal{V})=\sum_{\lambda\vdash n}a_\lambda\,f_\lambda}.$$
Exploiting the fact that the trace of conjugate matrices are equal, we may rewrite \pref{def_frobsimple} as
\begin{equation}
   \bleu{\mathcal{V}(\mbf{w})=\sum_{\mu\vdash n} \chi(\mu)\, \frac{p_\mu(\mb{w})}{z_\mu}},
\end{equation}
where $\chi(\mu)$ stands for ${\rm Trace}(M_\sigma)$, for any permutation $\sigma$ of cycle type $\mu$; and where
    $$\bleu{z_\mu:=1^{d_1}d_1!\, 2^{d_2}d_2!\cdots n^{d_n}d_n!},$$
 if $\mu$ has $d_i$ parts of size $i$.

\subsection*{An example}
To illustrate this notion of Frobenius characteristic, let us consider the  ring $\bleu{\Ring_n:=\C[\mb{x}]}$ of polynomials in the $n$ variables $\mb{x}=x_1,x_2,\ldots,x_n$;  which affords the linear basis of \defn{monomials} 
    $$\bleu{\mb{x}^\mb{a}:=x_1^{a_1}x_2^{a_2}\cdots x_n^{a_n}},$$
 indexed by $n$-vectors $\mb{a}=(a_1,a_2,\ldots, a_n)$ in $\N^n$.
This ring is graded with respect to the \defn{degree}
\begin{eqnarray*}
    \bleu{\deg(\mb{x}^\mb{a})}&:=&\bleu{|\mb{a}|}\\
                      &:=&\bleu{a_1+a_2+\ldots +a_n},
   \end{eqnarray*}
 so that we have the decomposition
    $$\bleu{\Ring_n=\bigoplus_{d\geq 0} \Ring_n^{(d)}},$$
 where $\Ring_n^{(d)}$ stands for the \defn{homogeneous component} of degree $d$ of $\Ring_n$. In other words, $\Ring_n^{(d)}$ is the span of the monomials of degree $d$. The \defn{Hilbert series} of  $\Ring_n$ records the dimension of each of these component in the form of the power series
   $$\bleu{\Ring_n(q):=\sum_{d\geq 0} \dim(\Ring_n^{(d)})\, q^d}.$$
It is no hard to show directly that
   \begin{equation}
      \bleu{\Ring_n(q)=\left(\frac{1}{1-q}\right)^n},\qquad \hbox{or equivalently}\qquad \bleu{\dim(\Ring_n^{(d)})=\binom{n+d-1}{d}}.
   \end{equation}
However, we will actually prove a more general formula. 

The symmetric group acts on $\Ring_n$ by permutation of the variables:
   $$\bleu{\sigma\cdot x_i:=x_{\sigma(i)}}.$$
Although $\Ring_n$ is infinite dimensional, each of the component $\Ring_n^{(d)}$ is finite dimensional. We exploit this to define its \defn{graded Frobenius characteristic} 
    \begin{equation}\label{def_frob_ring}
        \bleu{\Ring_n(\mb{w};q):=\sum_{d\geq 0} \Ring_n^{(d)}(q) q^d}. 
    \end{equation}
This makes sense, since the action of $\S_n$ respects degree. In fact, this is a permutation action on the basis of monomials.
From the definition, we may calculate an explicit formula for $\Ring_n(\mb{w};q)$ as follows. A straightforward reformulation of \pref{def_frob_ring} gives 
\begin{equation}\label{def_frob_ring_pmu}
   \bleu{\Ring_n(\mb{w};q)=\sum_{\mu\vdash n} \sum_{\sigma_\mu \cdot \mb{x}^\mb{a}=\mb{x}^\mb{a}} q^{|\mb{a}|}\, \frac{p_\mu(\mb{w})}{z_\mu}},
 \end{equation}
 with $\sigma_\mu$ any fixed permutation of cycle type $\mu$.
We thus have to count monomials, such that $\sigma_\mu \cdot \mb{x}^\mb{a}=\mb{x}^\mb{a}$, with weight $q^{|\mb{a}|}$. 
To do this, we select $\sigma_\mu$ to be
   $$\bleu{\sigma_\mu=\rouge{(}1,\rouge{\ldots},\mu_1\rouge{)}\,\rouge{(}\mu_1+1,\rouge{\ldots},\mu_1+ \mu_2\rouge{)}\, \cdots\,
                            \rouge{(}\mu_1+\ldots +\mu_{k-1}+1,\rouge{\ldots}, \mu_1+\ldots +\mu_k\rouge{)}},$$
if $\mu=\mu_1\mu_2\ldots \mu_k$. Clearly, $\bleu{\sigma_\mu \cdot \mb{x}^\mb{a}=\mb{x}^\mb{a}}$ if and only if 
$$\begin{array}{lll}
          &\bleu{a_1= \cdots = a_{\mu_1}} &\rouge{(=b_1)},\\
          & \bleu{a_{\mu_1+1}= \cdots = a_{\mu_1+\mu_2}} &\rouge{(=b_2)},\\
          & \qquad \bleu{\vdots}\\
          &\bleu{a_{\mu_1+\ldots +\mu_{k-1}+1}= \cdots = a_{\mu_1+\ldots +\mu_k}} &\rouge{(=b_k)}.
 \end{array}$$
 The coefficient of $p_\mu(\mb{w})/z_\mu$ in \pref{def_frob_ring_pmu} is thus seen to be equal to
   $$\bleu{\sum_{(b_1,\ldots,b_k)\in \N^k} q^{b_1\mu_1+b_2\mu_2+\ldots b_k\mu_k}
         =\prod_{i=1}^k \frac{1}{1-q^{\mu_i}}},$$
and we get the final formula
\begin{equation}\label{frob_ring_pmu}
   \bleu{\Ring_n(\mb{w};q)=\sum_{\mu\vdash n} \frac{1}{z_\mu}  \prod_{i=1}^k \frac{p_{\mu_i}(\mb{w})}{1-q^{\mu_i}}  }.
\end{equation}
The resulting expressions may be expanded in terms of Schur functions, and we calculate that
  \begin{eqnarray*}
     \Ring_1(\mb{w};q)&=&\frac{1}{(1-q)}\, s_1(\mb{w}),\\[4pt]
     \Ring_2(\mb{w};q)&=&\frac{1}{(1-q)(1-q^2)}\left( s_2(\mb{w})+\rouge{q}\,s_{11}(\mb{w})\right),\\[4pt]
     \Ring_3(\mb{w};q)&=&\frac{1}{(1-q)(1-q^2)(1-q^3)}\left( s_3(\mb{w})+\rouge{(q+q^2)}s_{21}(\mb{w})+\rouge{q^3}\,s_{111}(\mb{w})\right),\\[4pt]
      \Ring_4(\mb{w};q)&=&\frac{1}{(1-q)(1-q^2)(1-q^3)(1-q^4)}\big( s_4(\mb{w})+\rouge{(q+q^2+q^3)}s_{31}(\mb{w})\\[4pt]
                                  &&\qquad +\rouge{(q^2+q^4)}s_{22}(\mb{w})+\rouge{(q^3+q^4+q^5)}s_{211}(\mb{w})+\rouge{q^6}\,s_{1111}(\mb{w})\big).\\
 \end{eqnarray*}
The nice numerators of these expressions are in fact special instances of the well known \defn{Hall-Littlewood polynomials}. The coefficients of their expansion in the Schur basis are known as the \defn{$q$-Kostka polynomials}. The fact that these lie in $\N[q]$ was explained combinatorially by Lascoux and SchŸtzenberger (see~\cite{lascoux}), in terms of the ``charge statistic'' on standard tableaux.

\subsection*{Parking function representation} For a fixed integer partition $\lambda$, consider the transitive permutation action of $\S_n$ on the set of \defn{$\lambda$-partitions} of $\{1,2,\ldots,n\}$, i.e.: having parts size specified by $\lambda$. For example, with $\lambda=32$, the corresponding set contains the $10$ partitions
 $$  \begin{matrix}
      \part{1,2,3}{4,5}, & \part{1,2,4}{3,5}, & \part{1,2,5}{3,4}, & \part{1,3,4}{2,5},\\[4pt]
      \part{1,3,5}{2,4},  &  \part{1,4,5}{2,3}, &  \part{2,3,4}{1,5}, &  \part{2,3,5}{1,4},\\[4pt]
      \part{2,4,5}{1,3},   &  \part{3,4,5}{1,2}.
       \end{matrix} $$
 Then, one checks that the associated frobenius characteristic is $h_\lambda(\mbf{w})$. This is can be done by a straightforward calculation; but one may also first  observe that \pref{trivial} states that this holds for $n$-partitions, and then judiciously apply \pref{tensor} to get the general case. The dimension of this representation is clearly given by the multinomial coefficient
   $$\bleu{\binom{n}{\lambda} = \frac{n!}{\lambda_1!\lambda_2!\cdots \lambda_k!}}.$$
 
For a given $r$-Dyck path $\alpha$ of height $n$, it easy to see that the the action of $\S_n$, on the set of parking functions of shape $\alpha$, is isomporphic to the action of $\S_n$ on $\lambda$-partitions. Here $\lambda$ is the partition obtained by sorting the parts of the composition $\gamma(\alpha)$ in decreasing order. Modulo some calculations, we have (see ~\cite{stanley_shi}).
  \begin{proposition}
   \bleu{The Frobenius characteristic of the space $\mathcal{P}^{(r)}_n$, spanned by $r$-parking functions of height $n$, my expressed in the following equivalent manners}
   \begin{eqnarray}
     \bleu{\mathcal{P}^{(r)}_n(\mbf{w})} 
     &=&\bleu{\sum_{\alpha\in \Dyck{r}{n}}   h_{\gamma(\alpha)}(\mbf{w})}\label{park_h}\\
      &=&\bleu{\frac{1}{rn+1}\,\sum_{\lambda\vdash n} (rn+1)^{\ell(\lambda)}\, \frac{p_\lambda(\mbf{w})}{z_\lambda} }\\   
      &=&\bleu{\frac{1}{rn+1}\,\sum_{\lambda\vdash n} \prod_i \binom{rn+\lambda_i}{n}\,m_\lambda(\mbf{w})} \\
      &=&\bleu{\frac{1}{rn+1}\,\sum_{\lambda\vdash n} s_\lambda(1^{rn+1})\,s_\lambda(\mbf{w})} \\
      &=&\bleu{ \sum_{\lambda\vdash n} \frac{rn\,(rn-1)\cdots (rn-\ell(\lambda)+2)}{d_1(\lambda)!\cdots  d_n(\lambda)! }\,h_\lambda(\mbf{w})},
   \end{eqnarray} 
 \bleu{where $d_i(\lambda)$ is the number of parts of size $i$ in $\lambda$}.
\end{proposition}
Recall that $s_\lambda(1,1,\ldots,1)$ may easily be calculated using the formula
\begin{equation}\label{evSchurun}
    \bleu{s_\lambda(\underbrace{1,\ldots,1}_{\rouge{N\ {\rm copies}}})=   \prod_{(i,j)\in\lambda}\frac{\rouge{N}+j-i}{h_{ij}}},
\end{equation}
  where $h_{ij}$ stands for the hook length of the cell $(i,j)$ in $\lambda$ (here identified with its Ferrers diagram). Thus, we may calculate that the respective multiplicities of the trivial and the sign representation in $\mathcal{P}^{(r)}_n$, are given by the Fuss-catalan numbers
     \begin{equation}
        \bleu{\frac{1}{rn+1}\,h_n(1^{rn+1})=\frac{1}{rn+1}\,\binom{(r+1)\,n}{n}},
       \end{equation}
 and
      \begin{equation}
        \bleu{\frac{1}{rn+1}\,e_n(1^{rn+1})=\frac{1}{(r-1)n+1}\,\binom{r\,n}{n}}.
       \end{equation}
In the sequel we will see that a ``sign-twisted'' version of the Frobenius characteristic $\mathcal{P}^{(r)}_n(\mbf{w})$ plays an important role in the study of ``diagonal harmonics''. 
Technically, this sign-twisting simply means that we replace $ h_{\gamma(\alpha)}(\mbf{w})$ by $e_{\gamma(\alpha)}(\mbf{w})$ in \pref{park_h}.
In fact, a weighted enumeration of  semi-standard tableaux of shape $(\lambda+1^n)/\lambda$ is of interest in that context, and we will be considering the expression
\begin{equation}\label{dinvfrob}
   \bleu{\mathcal{D}_n(\mbf{w};q,t) :=\sum_{\lambda\subseteq \delta_n} \sum_{\tau} q^{|\delta_n/\lambda|} t^{\dinv(\tau)} \mbf{w}_\tau},
\end{equation}
where the second sum runs over semi-standard tableaux $\tau$, of shape $(\lambda+1^n)/\lambda$, and $\mbf{w}_\tau$ is the \defn{tableau monomial}
   $$\bleu{ \mbf{w}_\tau :=\prod_{(i,j)\in (\lambda+1^n)/\lambda} w_{\tau(i,j)} },$$
  in the variables $\mbf{w}$. Observe that $|\delta_n/\lambda|$, which is the number of cells in the skew shape $\delta_n/\lambda$, is simply the area of the Dyck path corresponding to $\lambda$.
  
We tie this with the current section by observing, as discussed in  \cite{HHLRU}, that we have
\begin{equation}
   \bleu{\sum_{\tau}  \mbf{w}_\tau=e_{\gamma(\lambda)}(\mbf{w})}.
\end{equation}
Hence,
  \begin{equation}\label{sign_twist_Pn}
   \bleu{\mathcal{D}_n(\mbf{w};q,1) =\sum_{\lambda\subseteq \delta_n}   q^{|\delta_n/\lambda|} e_{\gamma(\lambda)}(\mbf{w})}.
\end{equation} 
Along these lines, let us also mention that Theorem 3.1.3 of~\cite{HHLRU} states that 
\begin{equation}\label{LLLpos}
   \bleu{\mathcal{D}_n^\lambda(\mbf{w};t):=\sum_{\newatop{\tau:(\lambda+1^n)/\lambda\longrightarrow \N}{\tau\ \hbox{\tiny semi-standard}} }  t^{\dinv(\tau)} \mbf{w}_\tau},
\end{equation}
is symmetric in the variables $\mbf{w}$, and that it is \defn{Schur positive}. This is to say that $\mathcal{D}_n^\lambda(\mbf{w};t)$ lies in $\N[t]\{s_\lambda(\mbf{w})\ |\ \lambda\vdash n\}$.

\subsection*{Parking and interval representation}    
Just as for parking functions, the symmetric group acts on pairs in the set $\mathcal{Q}_n^{(r)}$ (see \pref{ensQ}).
We have the following.
  \begin{proposition}
   \bleu{The Frobenius characteristic of the space $\mathcal{Q}^{(r)}_n$ {\rm (}spanned by pairs $(\alpha,\varphi)$, with $\varphi$ a $r$-parking functions of height $n$, and $\alpha\leq \iota(\varphi)$ in $r$-Tamari order{\rm )}, may be expressed in the following equivalent manners}
   \begin{eqnarray}
     \bleu{\sum_n\mathcal{Q}^{(r)}_n(\mbf{w})} 
     &=&\bleu{\sum_{n\geq 0} \sum_{\alpha,\beta\in\Dyck{r}{n}}    \chi(\alpha\leq\beta)\, h_{\gamma(\beta)}(\mbf{w})}\label{formun}\\
      &=&\bleu{ \frac{1}{(rn+1)^2} \exp\!\left(\sum_{k\geq 1} (r\,n+1)\binom{(r+1)\,k}{k}\,\frac{p_k(\mbf{w})}{k} \right) }\label{formdeux}\\   
      &=&\bleu{ \frac{1}{(rn+1)^2} \sum_{\lambda\vdash n}  \prod_{j\in \lambda} \frac{r n+1}{r n+j+1}
      \binom{(r+1)(rn+j+1)}{j}\, m_\lambda(\mbf{w})}.\label{formtrois}
   \end{eqnarray} 
\end{proposition}
\begin{proof}[\bf Proof.] Equality between \pref{formun} and \pref{formdeux} is shown in \cite{chapuy_polya}. To check that \pref{formdeux} equals \pref{formtrois}, we start with the generating function identity (established below)
\begin{equation}\label{fonct_gen}
   \bleu{\exp\!\left(\sum_{k\geq 1} a\,\binom{b\,k}{k}\frac{t^k}{k}\right) =
  1+  \sum_{j\geq 1}  {ab}\,\rouge{(}ab+jb\rouge{)_{(j)}}\, \frac{t^j}{j!}},
   \end{equation}
where we use the (shifted) Pochhammer symbol notation
    $$\rouge{(}u\rouge{)_{(j)}}:=(u-1)\,(u-2)\,\cdots \, (u-j+1).$$
In light of \pref{genh} and \pref{genp}, one way of parsing Formula \pref{fonct_gen} is to consider it as stating that $\rouge{h_j(\mbf{y})}$ evaluates to \bleu{${ab}\,\rouge{(}ab+jb\rouge{)_{(j-1)}}$}, when  $\rouge{p_k(\mbf{y})}$ is replaced by  \bleu{$a\binom{b\,k}{k}$}. On the other hand, using Cauchy identity \pref{Cauchy}, we have that
    $$\exp\!\left(\sum_{k\geq 1} \rouge{p_k(\mbf{y})}\,\frac{p_k(\mbf{w})}{k}\right)= 
       \sum_{\lambda} \rouge{h_\lambda(\mbf{y})}\,m_\lambda(\mbf{w}).$$
Hence, setting $a=rn+1$ and $b=r+1$, we get
\begin{eqnarray*}
    &&\exp\!\left(\sum_{k\geq 1} \bleu{ (r\,n+1)\binom{(r+1)\,k}{k}}\,\frac{p_k(\mbf{w})}{k}\right)=\\ 
     &&\qquad \qquad   \sum_{\lambda} \bleu{\prod_{j\in\lambda} (r+1)\,(rn+1)\,\rouge{(}(r+1)(rn+j+1)\rouge{)_{(j)}}}\,m_\lambda(\mbf{w}).
   \end{eqnarray*}
   Since it is straightforward that
      $$(r+1)\,(rn+1)\,\rouge{(}(r+1)(rn+j+1)\rouge{)_{(j)}} = \frac{r n+1}{r n+\rouge{j}+1}
      \binom{(r+1)(rn+\rouge{j}+1)}{\rouge{j}},$$
   we have finished our proof.
\end{proof}
To establish \pref{fonct_gen}, we proceed as follows.
For any positive integer $b$, consider the series (obtained by setting $a=1$ in the left-hand side of \pref{fonct_gen}) 
 $$\bleu{Z(t):=\exp\!\left(\sum_{k\geq 1} \binom{b\,k}{k}\frac{t^k}{k}\right)}.$$
We first show that $Z(t)$ satisfies the equation
\begin{equation}\label{eqalg}
    \bleu{Z(t)=(1+t\,Z(t))^b}.
 \end{equation}
Equivalently, the logarithmic derivative of $Z(t)$ is such that
\begin{equation}\label{logder}
    \bleu{\frac{Z'(t)}{Z(t)}=\frac{b\,Z(t)}{1-(b-1)\,t\,Z(t)}},
 \end{equation}
 and we want to check that this is equal to
    $$\bleu{Y(t):=\sum_{k\geq 1} \binom{b\,k}{k}\,t^k}.$$
Now, it happens that $Y(t)$ satisfies the equation
    \begin{equation}
       \bleu{Y(t)=\frac{(b\,(1+t\,Y(t)))^b}{\rouge{(}1+(b-1)(1+t\,Y(t))\rouge{)}^{b-1}}}.
     \end{equation}
Replacing $Y(t)$ by the right-hand side of \pref{logder} in this last equality, after calculations, we see that  $Z(t)$ satisfies \pref{eqalg}. 
Using Lagrange inversion, it is then easy to check that the right-hand side of \pref{fonct_gen} is indeed equal to $Z(t)^a$.

 \section{Diagonal harmonics}\label{trivariate_diag}
In the  ring $\bleu{\Ring_n:=\C[X]}$ of polynomials in the three sets of variables
   \begin{displaymath}\bleu{X:=\begin{pmatrix}
      x_1 & x_2 & \cdots &x_n\\
      y_1& y_2 & \cdots &y_n\\
      z_1&  z_2 & \cdots &z_n
      \end{pmatrix}, }\end{displaymath}
  ${\Harm_n}$ is the subspace of polynomial zeros of the 
\defn{polarized power sums} differential operators
  \begin{equation}\label{def_power_sum}
      \bleu{P_\alpha(\partial X):=\sum_{j=1}^n\del{x_j}^a\del{y_j}^b \del{z_j}^c},
  \end{equation}
for $\alpha=(a,b,c)$ running thought the set of vectors of \defn{norm} $\bleu{|\alpha|:=a+b+c}$, with $1\leq |\alpha|\leq n$. Here,  $\partial v$ denotes derivation with respect to $v$. It is clear from the definition that $\Harm_n$ is closed under derivation. 
For one set of variables, say $\mb{x}=x_1,\ldots,x_n$,  $\Harm_n$ is entirely described by classical theorems (see~\cite[section 3.6]{humphreys}) as the span of all partial derivatives of the Vandermonde determinant
     \begin{displaymath}\Delta_n(\mb{x})=\det\left( x_i^j\right)_{1\leq i\leq n,\ 0\leq j\leq n-1},\end{displaymath}
It is of dimension $n!$, and isomorphic to the regular representation of $\S_n$, for the action that permutes variables. The description of the case of two sets of variables took close to 15 years to be finalized (see~\cite{haimanhilb}), and much remains to be understood. It is of dimension $(n+1)^{n-1}$. It has been observed by Haiman (see \cite{JAC}) that the dimension of $\Harm_n$ seems to be given by the formula
\begin{equation}\label{dim_Hn}
   \bleu{\dim \Harm_n= 2^n\,(n+1)^{n-2}},
\end{equation}
but, until very recently, almost no one had further studied the trivariate case. We have endeavored to do so, considering more general spaces $\bleu{\Harm_{n}^{(r)}}$ (see definition~\pref{def_higher}) for which we have experimentally found out that 
\begin{equation}\label{dim_Hrn}
   \bleu{\dim\Harm_{n}^{(r)}= (r+1)^n\,(r\,n+1)^{n-2}}.
\end{equation}

\subsection*{Symmetric group action}
We turn both $\Ring_n$ and $\Harm_n$ into $\S_n$-module  by considering the \defn{diagonal action} of symmetric group $\S_n$ on  variables.  Recall that, for $\sigma\in \S_N$, the polynomial $\sigma\cdot f(X)$ is obtained by replacing the variables in $f(X)$ by
     \begin{displaymath}\begin{pmatrix}
      x_{\sigma(1)} & x_{\sigma(2)} & \cdots &x_{\sigma(n)}\\
      y_{\sigma(1)}& y_{\sigma(2)} & \cdots &y_{\sigma(n)}\\
      z_{\sigma(1)}&  z_{\sigma(2)} & \cdots &z_{\sigma(n)}
      \end{pmatrix}. \end{displaymath}
We denote by  $X^A$ the monomial
    \begin{displaymath}\bleu{X^A:=(x_1^{a_1}\cdots x_n^{a_n})\cdot (y_1^{b_1}\cdots y_n^{b_n})\cdot (z_1^{c_1}\cdots z_n^{c_n})},\end{displaymath} 
 for a matrix of integers
    \begin{displaymath}A=\begin{pmatrix}
      a_1 & a_2 & \cdots &a_n\\
      b_1& b_2 & \cdots &b_n\\
      c_1&  c_2 & \cdots &c_n
      \end{pmatrix}. \end{displaymath}
Writing $A_j$ for the  $j^{\rm th}$-column of $A$, we may define
the \defn{degree} of a monomial $X^A$ to be the vector 
    \begin{displaymath}\bleu{\deg(X^A):=\sum_{j=1}^n A_j},\end{displaymath} 
 in $\N^3$. Clearly the diagonal action preserves degree.
 The \defn{total degree} ${\rm tdeg}(X^A)$ of a monomial $X^A$ is the sum of the components of $\deg(X^A)$.
An interesting subspace of $\Harm_n$ is the space $\bleu{\Alt_n}$ of \defn{alternating} polynomials in $\Harm_n$. Recall that these are the polynomials such that $\sigma\cdot f(X)=\sign(\sigma)\,f(X)$. It is also worth recalling that, in the case of two sets of variables, this subspace has dimension equal to the Catalan number
 \bleu{$C_n$}. 
 Another observation of Haiman ({\sl loc. sit.}) is that, in the trivariate case, we seem to have
\begin{equation}\label{dim_An}
   \bleu{\dim\Alt_n=\frac {2}{ n\,( n+1) } \binom{4\,n+1}{ n-1}}.
\end{equation}
More generally, our experimental calculations suggest that
 the  alternating component $ \Alt_n^{(r)}$ of $ \Harm_n^{(r)}$ has dimension
\begin{equation}\label{dim_Altr}
   \bleu{ \dim \Alt_n^{(r)}}=\bleu{\frac {( r+1)}{ n\,( r n+1) } \binom{( r+1) ^{2}\,n+r}{ n-1} }.
\end{equation}

\subsection*{Degree grading} The space $\Harm_n$ is graduated by degree, and its \defn{homogeneous components}  $\bleu{\Harm_{n,\mb{d}}}$ (with $d\in \N^3$) are $S_n$-invariant, hence ${\Alt_n}$ is also graded. There corresponds direct sum decompositions
\begin{equation}\label{decomp_hom}
   \bleu{ \Harm_n=\bigoplus_{d\in \N^3} \Harm_{n,\mb{d}}},\qquad {\rm and}\qquad  \bleu{ \Alt_n=\bigoplus_{d\in \N^3} \Alt_{n,\mb{d}}}.
 \end{equation} 
A striking fact is that the space $\Harm_n$ is finite dimensional. In fact, the homogeneous component $\Harm_{n,\mb{d}}$ is non vanishing only if 
    \begin{displaymath}|\mb{d}|=d_1+d_2+d_3\leq \binom{n}{2},\end{displaymath}
where $\mb{d}=(d_1,d_2,d_3)$.  We thus have the {Hilbert series (polynomials)}:
   \begin{equation}\label{defn_Hilb}
        \bleu{\Harm_n(\mb{q}):=\sum_{|\mb{d}|=0}^{\binom{n}{2}} \dim(\Harm_{n,\mb{d}})\, \mb{q}^\mb{d}}, \qquad{\rm and}\qquad
         \bleu{\Alt_n(\mb{q}):=\sum_{|\mb{d}|=0}^{\binom{n}{2}} \dim(\Alt_{n,\mb{d}})\, \mb{q}^\mb{d}}
    \end{equation}
  of $\Harm_n$ and $\Alt_n$,  writing $\mb{q}^\mb{d}$ for $q_1^{d_1}q_2^{d_2}q_3^{d_3}$. The Hilbert series $\Harm_n(\mb{q})$ is a symmetric polynomial in $q_1$, $q_2$ and $q_3$. In fact, using representation theory of $GL_3$, one can see that it is always Schur positive. This is to say that it expands as a positive integer coefficient linear combination of the Schur polynomials $s_\mu(\mb{q})$ in the variables $\mb{q}=q_1,q_2,q_3$. 
 
For example, we easily calculate that $\Harm_2$ affords the linear basis
\begin{equation}\label{base_2}
   \{1,x_2-x_1,y_2-y_1,z_2-z_1\},
 \end{equation}
whose subset $ \{x_2-x_1,y_2-y_1,z_2-z_1\}$ is clearly a basis of  $\Alt_2$.
We thus get  the Hilbert series
   \begin{displaymath}\Harm_2(\mb{q})=1+(q_1+q_2+q_3)=1+s_{1}(\mb{q}),\qquad {\rm and}\qquad \Alt_2(\mb{q})=q_1+q_2+q_3=s_{1}(\mb{q}).\end{displaymath}
To calculate $\Harm_n(\mb{q})$ for larger $n$, we exploit the fact $\Harm_n$ is closed under taking derivatives, as well as applying the operators
   \begin{displaymath}\bleu{E_{\mb{u}\mb{v}}^{(k)}:= \sum_{i=1}^n u_i\del{v_i}^k},\qquad (E_{\mb{u}\mb{v}}:=E_{\mb{u}\mb{v}}^{(1)}). \end{displaymath}
Here $\mb{u}$, $\mb{v}$ stand for any two of the three sets of $n$ variables:
    \begin{displaymath}\mb{x}=x_1,\ldots,x_n,\qquad \mb{y}=y_1,\ldots,y_n,\qquad {\rm and}\qquad \mb{z}=z_1,\ldots,z_n.\end{displaymath}
It easy to verify that the Vandermonde determinant $\Delta_n(\mb{u})$ 
belongs to $\Harm_n$, whether $\mb{u}$ be $\mb{x}$, $\mb{y}$ of $\mb{z}$. 
Using all this, we can calculate that
  \begin{eqnarray*}
      \Harm_3(\mb{q})&=&1+2\,(q_1+q_2+q_3)\\ 
         &&+2\,(q_1^{2}+q_2^{2}+q_3^{2})+3\,(q_1q_2+q_1q_3+q_2q_3)\\
         &&+q_1^{3}+q_2^{3}+q_3^{3}+q_1^{2}q_2+q_1^{2}q_3+q_1q_2^{2}+q_1q_3^{2}+q_2^{2}q_3+q_2q_3^{2}+
q_1q_2q_3
\end{eqnarray*}
by checking directly that we have the following respective bases $\mathcal{B}_{d}$ for each $\Harm_{n,\mb{d}}$:
\begin{displaymath}\begin{array}{lllll}
\Base_{300}=\{\Delta_3(\mb{x})\}, \qquad   \Base_{200}=\{\del_{x_1}\Delta_3(\mb{x}),\del_{x_2}\Delta_3(\mb{x})\},&
 \Base_{100}=\{\del_{x_1}^2\Delta_3(\mb{x}),\del_{x_1}\del_{x_2}\Delta_3(\mb{x})\}, \\[4pt]
\Base_{030}=\{\Delta_3(\mb{y})\}, \qquad    \Base_{020}=\{\del_{y_1}\Delta_3(\mb{y}),\del_{y_2}\Delta_3(\mb{y})\},&
 \Base_{010}=\{\del_{y_1}^2\Delta_3(\mb{y}),\del_{y_1}\del_{y_2}\Delta_3(\mb{y})\}, \\[4pt]
\Base_{003}=\{\Delta_3(\mb{z})\}, \qquad    \Base_{002}=\{\del_{z_1}\Delta_3(\mb{z}),\del_{z_2}\Delta_3(\mb{z})\},&
 \Base_{001}=\{\del_{z_1}^2\Delta_3(\mb{z}),\del_{x_1}\del_{z_2}\Delta_3(\mb{z})\}, \\[4pt]
\Base_{210}=\{E_{\mb{y}\mb{x}} \Delta_3(\mb{x})\}, \qquad  \Base_{120}=\{E_{\mb{y}\mb{x}}E_{\mb{y}\mb{x}} \Delta_3(\mb{x})\}, &
\Base_{201}=\{E_{\mb{z}\mb{x}} \Delta_3(\mb{x})\},   \\[4pt]
 \Base_{102}=\{E_{\mb{z}\mb{x}}E_{\mb{z}\mb{x}} \Delta_3(\mb{x})\}, \qquad 
 \Base_{021}=\{E_{\mb{z}\mb{y}} \Delta_3(\mb{y})\}, & \Base_{012}=\{E_{\mb{z}\mb{y}}E_{\mb{z}\mb{y}} \Delta_3(\mb{y})\},\\[4pt]
\Base_{110}=\{E_{\mb{y}\mb{x}}\del_{x_1}\Delta_3(\mb{x}),E_{\mb{y}\mb{x}}\del_{x_2}\Delta_3(\mb{x}), E_{\mb{y}\mb{x}}^{(2)} \Delta_3(\mb{x})\},\\[4pt] 
\Base_{101}=\{E_{\mb{z}\mb{x}}\del_{x_1}\Delta_3(\mb{x}),E_{\mb{z}\mb{x}}\del_{x_2}\Delta_3(\mb{x}), E_{\mb{z}\mb{x}}^{(2)} \Delta_3(\mb{x})\},\\[4pt] 
\Base_{011}=\{E_{\mb{z}\mb{y}}\del_{y_1}\Delta_3(\mb{y}),E_{\mb{z}\mb{y}}\del_{y_2}\Delta_3(\mb{y}), E_{\mb{z}\mb{y}}^{(2)} \Delta_3(\mb{y})\},\\[4pt] 
\Base_{000}=\{1\},\qquad \Base_{111}=\{E_{\mb{z}\mb{x}}E_{\mb{y}\mb{x}}\Delta_3(\mb{x})\}
\end{array}\end{displaymath}
Collecting previous observations, and doing some further explicit calculations, we get that
\begin{equation}\label{valeurs_Hilbert}
  \begin{array}{rcl}
 \Harm_{1}(\mb{q})&=&1,\\[3pt]
\Harm_{2}(\mb{q})&=&1+s_{1}(\mb{q})\\[3pt]
\Harm_{{3}}(\mb{q})&=&1+2\,s_{1}(\mb{q})+2\,s_{2}(\mb{q})+s_{{11}}(\mb{q})+
s_{{3}}(\mb{q}),\\[3pt]
\Harm_{{4}}(\mb{q})&=&1+3\,s_{1}(\mb{q})+5\,s_{2}(\mb{q})+3\,s_{{11
}}(\mb{q})+6\,s_{{3}}(\mb{q})+5\,s_{{21}}(\mb{q})+s_{{111}}(\mb{q})\\
    &&\qquad +5\,s_{{4}}(\mb{q})+4\,s_{{31}}(\mb{q})+3\,s_{{
5}}(\mb{q})+s_{{41}}(\mb{q})+s_{{6}}(\mb{q}).
\end{array}
\end{equation}
If we specialize one of the parameters (say $q_3$) to $0$, we get back the graded Hilbert series of bivariate diagonal harmonics (of overall dimension $(n+1)^{n-1}$) which has received a lot of attention in recent years. In other words, the $\mb{z}$-free (or $\mb{x}$-free, or $\mb{y}$-free) component of $\Harm_n$ coincides with the ``usual'' space of diagonal harmonics (in the bivariate case, as previously considered in the literature).
Hence, all of our formulas involving the parameters $\mb{q}$ specialize to known formulas by the simple device of setting one of the three parameters in $\mb{q}$ equal to $0$.
Evaluating $\Harm_n(\mb{q})$ at $q_i$ equal to $1$, we clearly get the overall dimension of $\Harm_n$. In other words, the resulting evaluation  of~\pref{valeurs_Hilbert} agrees with  Formula~\pref{dim_Hn}.

\subsection*{Graded character} We refine the dimension analysis by taking into account the decomposition into irreducibles of the homogeneous components of $\Harm_n$, see \pref{decomp_hom}. This is all encompassed into the \defn{graded Frobenius characteristic} of $\Harm$:
    \begin{displaymath}\bleu{\Harm_n^{(r)}(\mb{q};\mb{w}):=\sum_{d\in \N^3} \mb{q}^\mb{d}\, {\Harm_{n;\mb{d}}^{(r)}}(\mb{w})}.\end{displaymath}
For reasons already alluded to, the expansion of $\Harm_n(\mb{q};\mb{w})$ in terms of the $S_\lambda(\mb{w})$ affords Schur positive coefficients in the $s_\mu(\mb{q})$. This is to say that
   \begin{equation}
     \bleu{ \Harm_n^{(r)}(\mb{q};\mb{w})=\sum_{\lambda\vdash n} \left(\sum_{\mu} n_{\lambda,\mu}^{(r)} s_\mu(\mb{q})\right) S_\lambda(\mb{w})},\qquad n_{\lambda,\mu}\in\N.
   \end{equation}
Hence there are two kinds of Schur function playing a role here. To emphasize this, we denote those in the $\mb{q}$-variables by a lower case ``s''. For example, we have
\begin{equation}\label{valeurs_Frob}
\begin{array}{rcl}
\bleu{\Harm_{1}(\mb{w};\mb{q})}\hskip-8pt&=&\hskip-8pt \bleu{S_{1}(\mb{w})},\\[4pt]
\bleu{\Harm_{2}(\mb{w};\mb{q})}\hskip-8pt&=&\hskip-8pt \bleu{S_{2}(\mb{w})+\rouge{s_{1}(\mb{q})}\,S_{11}(\mb{w})},\\[4pt]
\bleu{\Harm_{3}(\mb{w};\mb{q})}\hskip-8pt&=&\hskip-8pt \bleu{S_{3}(\mb{w})+ \rouge{( s_{1}(\mb{q})+s_{2}(\mb{q}) )}\, S_{21}(\mb{w})+ 
                                                  \rouge{( s_{11}(\mb{q})+s_{3}(\mb{q}) )}\, S_{111}(\mb{w})},\\[4pt]
\bleu{\Harm_{4}(\mb{w};\mb{q})}\hskip-8pt&=&\hskip-8pt \bleu{S_{4}(\mb{w})+ \rouge{( s_{1}(\mb{q})+s_{2}(\mb{q})+s_{3} (\mb{q}))}\, S_{31}(\mb{w})}\\
         \hskip-8pt &  &\hskip-8pt \bleu{+ \rouge{( s_{2}(\mb{q})+s_{21}(\mb{q})+s_{4} (\mb{q}))}\, S_{22}(\mb{w})}\\
         \hskip-8pt &  &\hskip-8pt \bleu{+\rouge{( s_{11}(\mb{q})+s_{3}(\mb{q})+s_{21}(\mb{q})+s_{4}(\mb{q})+s_{31}(\mb{q})+s_{5}(\mb{q}) )}\, S_{211}(\mb{w})}\\
          \hskip-8pt &  &\hskip-8pt \bleu{+ \rouge{(s_{111}(\mb{q})+ s_{31}(\mb{q})+s_{41}(\mb{q})+s_{6}(\mb{q}) )}\, S_{1111}(\mb{w})}.
\end{array}
\end{equation}
The Schur functions $S_\lambda(\mb{w})$ act as placeholders for irreducible representations of $\S_n$.
Thus the multiplicities of the irreducible representations of $\S_n$ in $\Harm_{n}$  appear as the coefficients of the $S_{\mu}(\mb{w})$'s. Although we think that these coefficients  should evaluate in the three parameters $\mb{q}=q_1,q_2,q_3$, they are in fact ``universal'' expressions (see~\cite{bergeron_several}). This is to say that, in any $k$ sets of $n$ variables, $k=1,2,3,\ldots$, the above formulas evaluate to the ``correct'' value when we set $\mbf{q}=q_1,q_2,\ldots, q_k$.

Using Formula~\pref{evSchurun} we can exploit this, for $k$ generic, to get the following polynomial expressions in $k$, when we set all $q_i=1$. 
\begin{equation}\label{valeurs_Frob_k}
\begin{array}{rcl}
\bleu{\Harm_{1}(\mb{w};1^k)}\hskip-8pt&=&\hskip-8pt \bleu{S_{1}(\mb{w})},\\[4pt]
\bleu{\Harm_{2}(\mb{w};1^k)}\hskip-8pt&=&\hskip-8pt \bleu{S_{2}(\mb{w})+\rouge{k}\,S_{11}(\mb{w})},\\[4pt]
\bleu{\Harm_{3}(\mb{w};1^k)}\hskip-8pt&=&\hskip-8pt \bleu{S_{3}(\mb{w})+ \rouge{( k+\binom{k+1}{2} )}\, S_{21}(\mb{w})+ 
                                                  \rouge{( \binom{k}{2}+\binom{k+2}{3} )}\, S_{111}(\mb{w})}\\[4pt]
\bleu{\Harm_{4}(\mb{w};1^k)}\hskip-8pt&=&\hskip-8pt \bleu{S_{4}(\mb{w})+ \rouge{( k+\binom{k+1}{2}+\binom{k+2}{3})}\, S_{31}(\mb{w})},\\[4pt]
         \hskip-8pt &  &\hskip-8pt \bleu{
                                                         + \rouge{(\binom{k+1}{2}+2\binom{k+1}{3}+\binom{k+3}{4}))}\, S_{22}(\mb{w})}\\[4pt]
         \hskip-8pt &  &\hskip-8pt \bleu{+\rouge{( \binom{k}{2}+\binom{k+2}{3})+2\binom{k+1}{3}+\binom{k+3}{4}+3\binom{k+2}{4}+\binom{k+4}{5} )}\, S_{211}(\mb{w})}\\[4pt]
          \hskip-8pt &  &\hskip-8pt \bleu{+ \rouge{( \binom{k}{3}+ 3\binom{k+2}{4}+4\binom{k+3}{5}+\binom{k+5}{6} )}\, S_{1111}(\mb{w})}.
\end{array}
\end{equation}
Thus, if we set $k=1$, we get the Frobenius characteristic of the space of classical harmonics; at $k=2$, we get that of the space of bivariate diagonal harmonics; 
at $k=3$, trivariate diagonal harmonics; etc.

Considering only the coefficients of $S_{11\cdots1}(\mb{w})$, we obtain the formulas
\begin{equation}\label{valeurs_Alt_k}
\begin{array}{rcl}
\bleu{\Alt_{1}(1^k)} &=&  \bleu{1},\\[4pt]
\bleu{\Alt_{2}(1^k)} &=& \bleu{k},\\[4pt]
\bleu{\Alt_{3}(1^k)} &=& \bleu{ \binom{k}{2}+\binom{k+2}{3} },\\[4pt]
\bleu{\Alt_{4}(1^k)} &=&   \bleu{  \binom{k}{3}+ 3\binom{k+2}{4}+4\binom{k+3}{5}+\binom{k+5}{6} },
\end{array}
\end{equation}
which respectively give the appropriate values $1$, $C_n$, and \pref{dim_An};
at $k=1$, $k=2$ and $k=3$.

\subsection*{Bivariate diagonal harmonics} Before going on with our presentation, let us recall some classical facts about the spaces of harmonic polynomials, and recent results about the space of bivariate diagonal harmonic polynomials.  

The classical space of harmonic polynomials is $n!$ dimensional, with graded Frobenius given by the formula\footnote{This is the numerator of $\Ring_n(\mb{w};q)$ in Formula~\pref{frob_ring_pmu}).}
\begin{eqnarray}
    \bleu{\Harm_{n}(\mb{w};q)}&=&\bleu{\prod_{k=1}^n (1-q^k)\, \sum_{\lambda\vdash n}
                      s_\lambda(1,q,q^2,\ldots)\, S_\lambda(\mbf{w})},\\
               &=&\bleu{\prod_{k=1}^n (1-q^k)\, \sum_{\lambda\vdash n}
                       \prod_i \frac{p_{\lambda_i}(\mbf{w})}{1-q^{\lambda_i}}}.
  \end{eqnarray} 
  It follows that its Hilbert series is 
  $$\bleu{\Harm_{n}(q)=\prod_{k=1}^n \frac{1-q^k}{1-q}}.$$
 Observe also that when $q$ goes to $1$, we get
   $$\bleu{\Harm_{n}(\mb{w};1)= h_1^n}.$$
As already mentioned, the space of bivariate diagonal harmonic polynomials has dimension equal to $(n+1)^{n-1}$.
Its (bi-)graded Frobenius characteristic $\Harm_{n}(\mb{w};q,t)$ affords a description in terms of an operator which affords ``Macdonald Polynomials'' as eigenfunctions. Rather than recall the necessary notions, so that we could present this description, we recall from~\cite{HHLRU} a proposition giving a combinatorial description for $\Harm_{n}(\mb{w};q,1)$. Observe that we are here setting $t=1$. A combinatorial description of $\Harm_{n}(\mb{w};q,t)$ in full generality is still but a conjecture. We will come back to this in later sections. The notion of ``higher'' diagonal harmonics is explained in the next subsection.

\begin{proposition}[See \cite{HHLRU}{, Prop. 4.1.1, page 10}]
\bleu{The $q$-graded Frobenius characteristic of the space of bivariate (higher) diagonal harmonics is given by the formula}
\begin{equation}\label{frob_bivariate}
  \bleu{\Harm_{n}^{(r)}(\mb{w};q,1) = \sum_{\alpha\in \Dyck{r}{n}} q^{\area(\alpha)} \, e_{\gamma(\alpha)}(\mbf{w})}.
\end{equation}
\end{proposition}
This is the ``sign twisted'' version of the graded Frobenius for parking function representation,  expressed in Formula~\pref{sign_twist_Pn}.
We will see that a similar statement seems to hold in the trivariate case.

\subsection*{Higher harmonics} All of these considerations generalize to the spaces of \defn{higher diagonal harmonics}
\begin{equation}\label{def_higher}
   \bleu{\Harm_{n}^{(r)}:=(\Altid_n^{r-1})\cap (\Altid_n^{r-1}\Ideal_n)^\perp},
 \end{equation}
 where \bleu{$\Altid_n$} (resp.  \bleu{$\Ideal_n$}) stands  for the ideal generated by \defn{alternating} polynomials (resp. \defn{invariant}) polynomials  in $\Ring$, and the orthogonal complement is with respect to the $\S_n$-invariant scalar product defined by
\begin{equation}\label{def_scalar}
   \bleu{\scalar{X^A}{X^B}:=\begin{cases}
     A!& \text{if}\ A=B, \\
      0 & \text{otherwise}.
\end{cases}}
\end{equation}
Here $A!$ stands for the product of factorials:
    \begin{displaymath}\bleu{A!:=(a_1!\cdots a_n!)\cdot (b_1!\cdots b_n!)\cdot (c_1!\cdots c_n!)}.\end{displaymath}
On $\Harm_{n}^{(r)}$ the diagonal action is twisted by $\varepsilon^{r-1}$, where $\varepsilon$ stands for the sign representation. 
 For $n=2$, the space $\Harm_{2}^{(r)}$ decomposes as a direct sum of its isotypic \defn{invariant} component, $\bleu{\Inv_2^{(r)}}$, and  \defn{alternant} component, $\bleu{\Alt_2^{(r)}}$:
 \begin{equation}
      \bleu{\Harm_{2}^{(r)}=\Inv_2^{(r)}\oplus \Alt_2^{(r)}},
 \end{equation} 
 which afford the respective bases
\begin{eqnarray*}
    \bleu{\Base\Inv_2^{(r)}}&=&\bleu{\{ (x_2-x_1)^a(y_2-y_1)^b(z_2-z_1)^c\ |\ a+b+c=r-1 \}},\qquad {\rm and}\\
    \bleu{\Base\Alt_2^{(r)}}&=&\bleu{\{ (x_2-x_1)^a(y_2-y_1)^b(z_2-z_1)^c\ |\ a+b+c=r \}}. 
  \end{eqnarray*}
In fact, we have a universal formula, with both $k$ (where $\mbf{q}=q_1,q_2,\ldots, q_k$) and $r$ generic:
    \begin{displaymath}\bleu{\Harm_{2}^{(r)}(\mb{w};\mb{q})=\rouge{s_{r-1}(\mbf{q})}\, S_2(\mb{w}) + \rouge{s_{r}(\mbf{q})}\,  S_{11}(\mb{w})},\end{displaymath}
From this we get the universal Hilbert series 
  \begin{displaymath}\bleu{\Harm_{2}^{(r)}(\mb{q})=s_{r-1}(\mb{q}) + s_{r}(\mb{q}) },\quad {\rm and}\quad  \bleu{\Alt_2^{(r)}(\mb{q})=s_{r}(\mb{q})}.\end{displaymath}
In particular, setting all $q_i=1$, we have the dimension formulas
  \begin{displaymath}\bleu{\Harm_{2}^{(r)}(1^k)=\frac{k+2r-1}{k-1}\binom{k+r-2}{r}},\quad {\rm and}\quad  \bleu{\Alt_2^{(r)}(1^k)=\binom{k+r-1}{r}},\end{displaymath}
 as polynomials in the parameters $k$ and $r$.
For the trivariate case, this gives
  \begin{displaymath}\bleu{\dim \Harm_{2}^{(r)} = (r+1)^2 },\quad {\rm and}\quad  \bleu{\dim \Alt_{2}^{(r)} = \frac{(r+1)(r+2)}{2}}.\end{displaymath}
 which agree with Formulas~\pref{dim_Hrn} and \pref{dim_Altr}.
The above formulas are both encompassed in the following conjecture.
\begin{conjecture}\label{conj1}
The graded Frobenius characteristic of   $\Harm_n^{(r)}$    is given by the formula
 \begin{equation}\label{formule_trivariate} 
   \bleu{\Harm_n^{(r)}(\mb{w};q,1,1)=\sum_{\alpha, \beta\in\Dyck{r}{n}}  
             \chi(\alpha\leq \beta)\,q^{d(\alpha,\beta)}\, e_{\gamma(\beta)}(\mbf{w})}.
\end{equation}
\end{conjecture}
Observe that, at $q=1$, this is a sign-twisted version of \pref{formun}. For example, at $r=1$, calculating as in \pref{valeurs_Frob_k} we obtain the following formulas that include as special cases  both~\pref{frob_bivariate} and \pref{formule_trivariate}.
\begin{eqnarray*}
\bleu{\Harm_1(\mb{w};q,\rouge{1^k})}&=&\bleu{e_{{1}}(\mb{w})},\\
\bleu{\Harm_2(\mb{w};q,\rouge{1^k})}&=& \bleu{e_1^{2}(\mbf{w})+\rouge{( q+(k-1) )}\,e_{{2}}(\mbf{w})},\\
\bleu{\Harm_3(\mb{w};q,\rouge{1^k})}&=& \bleu{e_1^{3}(\mbf{w})+ \rouge{( {q}^{2}+(k+1)\,q+(k+4)(k-1)/2 )}\, e_{{2}}(\mbf{w})e_{{1}}(\mbf{w})}
\\ && \bleu{+  \rouge{( {q}^{3}+(k-1)\,{q}^{2}+(k+2)(k-1)\,q/2+(k-1)(k^2+4k-6)/6)}\, e_{{3}}(\mbf{w})}.
\end{eqnarray*}

\section{Combinatorial conjectures}
The climax of our story concerns explicit (but still conjectural) combinatorial formulas for the graded Frobenius characteristic of the spaces of bivariate and trivariate higher harmonic.  We first recall the following (which involves a $r$-refined version of \pref{dinvfrob}).
\begin{conjecture}[Shuffle Conjecture, see \cite{HHLRU}]
\bleu{The graded Frobenius characteristic of the space of bivariate diagonal higher harmonics is given by the formula}
\begin{equation}\label{conj_hhlru}
   \bleu{ \Harm_{n}^{(r)}(\mbf{w};q,t) =\sum_{\lambda\subseteq \delta^{(r)}_n}\sum_{\newatop{\tau:(\lambda+1^n)/\lambda\longrightarrow \N}{\tau\ \hbox{\tiny \rm semi-standard}} } q^{\area(\lambda)} t^{\dinv_r(\tau)} \mbf{w}_\tau},
\end{equation}
\bleu{where $\area(\lambda)$ stands for the number of cells of the skew shape $\delta_n^{(r)}/\lambda$.}
\end{conjecture}
Although we know that the left-hand side is symmetric in $q$ and $t$; at the moment of this writing, there is no direct explanation of this fact for the right-hand side from the combinatorial point of view. On the other hand, by definition we know that $ \Harm_{n}^{(r)}(\mbf{w};q,t)$ expands positively (with coefficients in $\N[q,t])$) in terms of Schur functions of the variables $\mbf{w}$. The fact that this is also the case for the right-hand side of \pref{conj_hhlru} follows from \pref{LLLpos}.

From our previous discussion, we may recast the above formula in terms of $r$-semi-parking functions.
In this terminology, the conjectured identity \pref{conj_hhlru} becomes
\begin{equation}\label{conj_hhlrubis}
   \bleu{ \Harm_{n}^{(r)}(\mbf{w};q,t) =\sum_{\beta\in \Dyck{r}{n}}\sum_{\varphi\in \mathcal{W}_\beta} q^{\area(\beta)} t^{\dinv_r(\varphi)} \mbf{w}_\varphi},
\end{equation}
where all notions are directly inherited through the translation from semi-standard path terminology. Our reason for all this is the ease the rest of our presentation.

\subsection*{Trivariate extension} The most striking aspect of the new developments presented here, is that a similar conjecture seems to hold for the trivariate case. 
The technical statement is as follows.
\begin{conjecture}[see \cite{trivariate}]\label{conj2}
\bleu{The graded Frobenius characteristic of the space of trivariate diagonal higher harmonics is given by the formula}
\begin{equation}\label{bergeron_ratelle}
   \bleu{ \Harm_{n}^{(r)}(\mbf{w};q,t,1) =\sum_{\alpha,\beta \in\Dyck{r}{n}}\sum_{\varphi\in \mathcal{W}_\beta} 
      \chi(\alpha\leq \beta)\,q^{d(\alpha,\beta)} t^{\dinv_r(\varphi)} \mbf{w}_\varphi}.
\end{equation}
\end{conjecture}
Observe that this conjecture is lacking a description of the third statistic that should appear here, rather than be specialized to $1$.

\subsection*{Quasi-symmetric functions} Both of the above conjectures involve infinite sums. These may be turned into finite expressions using the theory of ``quasi-symmetric functions''. We start with the observation that, for a fixed shape $(\lambda+1^n)/\lambda$, all semi-standard tableaux of some specific families have the same diagonal-inversion statistic. More precisely, the value of $\dinv_r(\tau)$ is entirely characterized once some inequalities are known to hold between pairs of entries of $\tau$, irrespective of the actual values of these entries. We will group together all such terms to obtain quasi-symmetric functions.  Both conjectures then take the form of finite sums of expressions involving quasi-symmetric functions, indexed by compositions associated to diagonal descent sets.

Recall that degree $d$ homogeneous component, of the ring $\Qsym$ of quasi-symmetric function, is spanned by the \defn{monomial quasi-symmetric functions} $M_\gamma(\mbf{w})$, which are indexed by compositions $\gamma=(c_1,\ldots,c_k)$ of $d$:
    $$\bleu{M_\gamma(\mbf{w}):=\sum_{\mbf{u}\subset_k \mbf{w}} \mbf{u}^\gamma},$$
where we write $\mbf{u}\subset_k \mbf{w}$ to say that $\mbf{u}$ consists in a choice of $k$ of the variables in $\mbf{w}$, these variables being ordered as they appear in $\mbf{w}$. In other words,
     $$\bleu{\mbf{u}^\gamma = w_{i_1}^{c_1}w_{i_2}^{c_2}\cdots w_{i_k}^{c_k}},\qquad{\rm with}\qquad \bleu{ i_1<i_2<\ldots<i_k}.$$
 For example,
     $$M_{21}(x,y,z)=x^2y+x^2z+y^2z,\qquad {\rm and}\qquad M_{12}(x,y,z)=xy^2+xz^2+yz^2.$$
 It easy to see that the monomial symmetric function $m_\lambda$ is equal to the sum of all $M_\gamma$, with $\gamma$ running through the set of all compositions that give $\lambda$ when their parts are reordered in decreasing order.
 
The so-called \defn{fundamental} basis (or $F$-basis) of $\Qsym$ is defined as the set of quasi-symmetric functions
     $$\bleu{F_\gamma(\mbf{w}):=\sum_{\alpha\preceq \gamma} M_\alpha(\mbf{w})},$$
 where the sum is over all compositions $\alpha$  that are \defn{finer} than $\gamma$. Recall that this is to say that $\gamma$ may be obtained from $\alpha$ by adding together consecutive parts. Thus, we have
\begin{eqnarray*}
   F_{3121}(\mbf{w})&=&M_{{3121}}(\mbf{w})+M_{{31111}}(\mbf{w})+M_{{21121}}(\mbf{w})+M_{{211111}}\\
      &&\qquad+M_{{12121}}(\mbf{w})+M_{{121111}}(\mbf{w})+M_{{111121}}(\mbf{w})+M_{{1111111}}(\mbf{w}).
  \end{eqnarray*}
 An alternate description exploits the classical bijection between partitions of $d$ and subsets of $\{1,2,\ldots,d-1\}$. Recall that this is the bijection that sends the composition $\gamma=(c_1,\ldots,c_k)$ to the set of partial sums 
    $$\bleu{S(\gamma):=\left\{\sum_{i=1}^j c_i\ \Big|\ 1\leq j\leq k-1\right\}}.$$
Thus we have $\alpha\preceq \gamma$ in refinement order if and only if $S(\gamma)\subset S(\alpha)$.

\subsection*{Finite form of conjectures} We can now collect terms in \pref{conj_hhlrubis} and \pref{bergeron_ratelle} so that they be expressed in terms of the $F$-basis. 
The resulting formulas take the form
\begin{eqnarray}
   \bleu{ \Harm_{n}^{(r)}(\mbf{w};q,t)} &=&\bleu{\sum_{\beta\in \Dyck{r}{n}}\sum_{\varphi\in \rouge{\Park_\beta}} q^{\area(\beta)} t^{\dinv_r(\varphi)} F_{\rouge{\theta(\varphi)}}(\mbf{w})},\label{shuff_quasi}\\
   \bleu{ \Harm_{n}^{(r)}(\mbf{w};q,t,1)} &=&\bleu{\sum_{\alpha,\beta \in\Dyck{r}{n}}\sum_{\varphi\in \rouge{\Park_\beta}} 
      \chi(\alpha\leq \beta)\,q^{d(\alpha,\beta)} t^{\dinv_r(\varphi)} F_{\rouge{\theta(\varphi)}}(\mbf{w})}.\label{shuff_tri_quasi}
\end{eqnarray}
with summation now being over the (finite) set of parking functions, and where \rouge{$\theta(\varphi)$} stands for the composition of $n$ encoding the diagonal descent set $\Desc_r(\varphi)$, of the parking function $\varphi$.

It is shown in \cite{HHLRU} that, for any $r$-Dyck path $\beta$, the sum 
   $$\bleu{D_\beta(\mbf{w};t):=\sum_{\varphi\in \Park_\beta} t^{\dinv_r(\varphi)} F_{\theta(\varphi)}(\mbf{w})}$$
expands, with coefficients in $\N[t]$, in terms of Schur functions in the variables $\mbf{w}$. Moreover, at $t=1$, we have
 $$\bleu{D_\beta(\mbf{w};1)=e_{\gamma(\beta)}(\mbf{w})}.$$
Thus, Conjecture~\ref{conj2} implies Conjecture~\ref{conj1}.

From the algebraic point of view, it is clear that we must have the identity
\begin{equation}
   \bleu{ \Harm_{n}^{(r)}(\mbf{w};q,1)= \Harm_{n}^{(r)}(\mbf{w};q,\rouge{0},1)}.
\end{equation}
It would thus be very interesting to find a direct combinatorial explanation (independent of the conjectures) of the equality that would have to hold between the corresponding right-hand sides of  \pref{shuff_quasi} and \pref{shuff_tri_quasi}.

\section{Other research directions}
Beside those explicitly considered in the text above, several natural questions arise. Chief among these is to develop a similar combinatorial setup for other finite reflection groups, for which much of the representation theory has been clarified in recent years (see~\cite{bergeron_several}). 

There has also been recently a lot of work on geometric (or polytopal) realizations of the associahedron (the Tamari poset for $r=1$). This leads to the natural question of describing similar constructions for all $r$-Tamari posets. Figure~\ref{tamari42} suggest a tantalizing outlook on this.

%%%%%%%%%%%%%%%%%%%%%%%%%%%%%%%%%%%%%%%%%%%%
%%%%%%%%%%%%%%%%%%%%%%%%%%%%%%%%%%%%%%%%%%%%

\end{document}